\documentclass{amsart}
\usepackage{amssymb,latexsym,amsmath,amsthm}
\usepackage{hyperref}
\usepackage[mathscr]{eucal}

\newtheorem{theorem}{Theorem}[section]
\newtheorem{lemma}[theorem]{Lemma}

\theoremstyle{definition}

\numberwithin{equation}{section}

\renewcommand{\l}{\lambda}

\newcommand{\g}{{\rm g}}
\newcommand{\dsq}{d^2\!}
\newcommand{\RR}{\ensuremath{\mathbb{R}}}

\newcommand{\prtl}{\ensuremath{\partial}}

\newcommand{\supp}{\ensuremath{\text{supp}}}
\newcommand{\veps}{\ensuremath{\varepsilon}}

\title[$L^q$ bounds on restricted spectral clusters]{$L^q$ bounds on restrictions of spectral clusters to submanifolds for low regularity metrics}
\thanks{The author was supported in part by the National Science
Foundation grant DMS-1001529.}

\author[M. D. Blair]{Matthew D. Blair}
\address{Department of Mathematics and Statistics, University of New Mexico,
  Albuquerque, NM 87131, USA}
\email{blair@math.unm.edu}

\begin{document}
\maketitle
\begin{abstract}
We prove $L^q$ bounds on the restriction of spectral clusters to submanifolds in Riemannian manifolds equipped with metrics of $C^{1,\alpha}$ regularity for $0 \leq \alpha \leq 1$. Our results allow for Lipschitz regularity when $\alpha =0$, meaning they give estimates on manifolds with boundary.  When $0< \alpha \leq 1$, the scalar second fundamental form for a codimension 1 submanifold can be defined, and we show improved estimates when this form is negative definite.  This extends results of Burq-G\'erard-Tzvetkov and Hu to manifolds with low regularity metrics.
\end{abstract}
\section{Introduction}
Let $M$ be a compact, smooth manifold of dimension $n \geq 2$ equipped with Riemannian metric $\g$ of at least Lipschitz regularity.  Let $\Delta_\g$ denote the associated (negative) Laplace-Beltrami operator whose action in coordinates is given by the differential operator
$$
\Delta_\g f = \frac{1}{\sqrt{\det \g_{kl}}} \sum_{i,j}\prtl_i\left(\g^{ij}\sqrt{\det \g_{kl}}\,\prtl_jf \right).
$$
There exists an orthonormal basis $\{ \phi_j\}_{j=1}^\infty$ of $L^2(M)$ consisting of eigenfunctions of $\Delta_\g$, which can be seen by passing to quadratic forms (see e.g. \cite[\S1]{SmC1alph}). We write the corresponding Helmholtz equation for $\phi_j$ as
$
(\Delta_\g  + \lambda_j^2)\phi_j =0
$
so that $\l_j$ gives the frequency of vibration associated to $\phi_j$.

Given $\l \geq 1$, we let $\Pi_\l$ be the projection operator on $L^2(M)$ defined by $
\Pi_\l f := \sum_{\l_j \in [\l, \l+1]} \langle f, \phi_j\rangle \phi_j,
$
where $\langle \cdot, \cdot \rangle $ denotes the usual $L^2$ inner product with respect to the Riemannian measure.  We call functions $f$ which are in the range of some $\Pi_\l$ ``spectral clusters".  They form approximate eigenfunctions or quasimodes as $\|(\Delta_\g  + \lambda^2)\Pi_\l f\|_{L^2(M)} \leq C\l \|f\|_{L^2(M)}$. In \cite{sogge88}, Sogge proved that when $\g$ is a $C^\infty$ metric, the following $L^q$ bounds on the projections $\Pi_\l f$ are satisfied for $q \geq 2$
\begin{equation}\label{soggebounds}
\|\Pi_\l f\|_{L^q(M)} \leq C \l^{\delta} \|f\|_{L^2(M)},
\end{equation}
where $\delta=\delta(q) = \max(\frac{n-1}{2}(\frac 12 - \frac 1q), n(\frac 12 -\frac 1q )-\frac 12)$.  He also provided examples showing that the exponent $\delta(q)$ is the best possible for these approximate eigenfunctions.  Since $\Pi_\l$ is a projection operator, any $L^q$ bound it satisfies implies $L^q$ bounds on individual eigenfunctions.  Determining when these bounds are sharp for subsequences of eigenfunctions is an area of active interest, though we do not examine this issue here.

In \cite{SmC2}, Smith proved that the bounds \eqref{soggebounds} are satisfied for $C^{1,1}$ metrics.  The assumption of $C^{1,1}$ regularity is the lowest degree of continuity needed to ensure the uniqueness of geodesics on $M$.  Since eigenfunctions naturally give rise to solutions to the wave equation, propagation of singularities suggests that this is a relevant consideration for the validity of such bounds.  Indeed, works of Smith-Sogge \cite{sS94} and Smith-Tataru \cite{smithtatexample} give examples of $C^{1,\alpha}$ metrics (Lipschitz when $\alpha =0$) which give rise to spectral clusters $\Pi_\l f_\l = f_\l$ for each $\l \geq 1$ such that
\begin{equation}\label{lowregexample}
\frac{\|f_\l\|_{L^q(M)}}{\|f_\l \|_{L^2(M)}} \geq c\l^{\frac{n-1}{2}(\frac 12-\frac 1q)(1+\sigma)}, \qquad \sigma = \frac{1-\alpha}{3+\alpha},
\end{equation}
showing that the bounds \eqref{soggebounds} cannot hold for $2 < q < \frac{2(n+2(1+\alpha)^{-1})}{n-1}$.  In each case, the cluster $f_\l$ is highly concentrated in a tube about a curve segment of length 1 and diameter $\l^{-\frac{2}{3+\alpha}}$ (cf. \eqref{alphatube} below).  This shows that the family $\{f_\l \}_{\l \geq 1}$ exhibits a greater degree of concentration than Sogge's examples which saturate the bounds \eqref{soggebounds} when $2<q\leq \frac{n-1}{2}(\frac 12 - \frac 1q)$ (they are concentrated in tubes with with diameter $\l^{-\frac 12}$).  In \cite{SmC1alph}, Smith showed positive results for any $C^{1,\alpha}$ metric, proving that the that the ratio on the left in \eqref{lowregexample} is always bounded above by $C\l^{\frac{n-1}{2}(\frac 12-\frac 1q)(1+\sigma)}$ when $2 \leq q \leq \frac{2(n+1)}{n-1}$.  He also proved that the bound \eqref{soggebounds} holds when $q=\infty$.  By interpolation, this shows \eqref{soggebounds} with a loss of $\sigma/q$ derivatives when $\frac{2(n+1)}{n-1} \leq q \leq \infty$, though a subsequent work of Koch-Smith-Tataru \cite{kstsharp} improves upon this.

In a similar vein, when $\g \in C^\infty$, results of Burq-G\'erard-Tzvetkov \cite{bgtrestr}, Hu \cite{Hu}, and Reznikov \cite{reznikov} show $L^q$ bounds on the restriction of these spectral clusters to embedded submanifolds $P \subset M$ of the form
\begin{equation}\label{restrictionbounds}
\|\Pi_\l f\|_{L^q(P)} \leq C \l^{\delta} \|f\|_{L^2(M)} \qquad q \geq 2
\end{equation}
where $\|\Pi_\l f\|_{L^q(P)}$ taken to mean the $L^q$ norm of the restriction $\Pi_\l f\big|_P$.  In this case, $\delta = \delta(k,q)$ depends on the dimension of the submanifold $k$ and on $q$.  In particular, when $k=n-1$, $\delta = \max(\frac{n-1}{2}-\frac{n-1}{q},\frac{n-1}{4}-\frac{n-2}{2q})$, that is,
\begin{equation}\label{deltadefhyper}
\delta(n-1,q) =
\begin{cases}
\frac{n-1}{2}-\frac{n-1}{q}, & \text{if } \frac{2n}{n-1} \leq q \leq \infty, \\
\frac{n-1}{4}-\frac{n-2}{2q}, & \text{if } 2 \leq q \leq \frac{2n}{n-1}
\end{cases}.
\end{equation}
Otherwise, when $1 \leq k \leq n-2$,
\begin{equation}\label{deltadef}
\delta(k,q) =
\frac{n-1}{2}-\frac kq
\end{equation}
with the exception of $(k,q)=(n-2,2)$ where there is a logarithmic loss for $\l \geq 2$, $\|\Pi_\l f\|_{L^2(P)} \leq C (\log \l)^{\frac 12} \l^{\frac 12}\|f\|_{L^2(M)}$.  These bounds were proved in a semiclassical setting by Tacy \cite{tacy}.  We also remark that the bound \eqref{restrictionbounds} in the case $k=n-1$, $q=2$ was previously observed by Tataru \cite{tatarubdry} as a consequence of the estimates in Greenleaf-Seeger \cite{greenleafseeger}.  As will be discussed in \S2, these bounds provide an improvement over what would be obtained by trace theorems for Sobolev spaces.

One reason the bounds \eqref{soggebounds}, \eqref{restrictionbounds} are of such great interest is that they illuminate the size and concentration properties of eigenfunctions.  In particular, Smith's work on $C^{1,\alpha}$ metrics \cite{SmC1alph} is significant in that it addresses concentration phenomena in situations where the roughness of the metric means that geodesic curves may fail to be unique.  It also led to the development of sharp bounds of the form \eqref{soggebounds} for the Dirichlet and Neumann Laplacians on compact Riemannian manifolds with boundary (see \cite{smithsogge06}).  Indeed, one strategy for proving estimates in this context is to form the double of the manifold, essentially gluing two copies of the manifold along the boundary.  While this eliminates the boundary, it gives rise to a metric of Lipschitz regularity (see e.g. \cite[p.420]{bssmulti}).  Hence any result on manifolds with Lipschitz metrics also applies to manifolds with boundary.  At the same time, the bounds \eqref{restrictionbounds} when $n=2$, $k=1$ (curves in 2 dimensional manifolds) for $\g \in C^{\infty}$ have garnered additional interest in recent works which relate improvements in these estimates to improvements in the inequalities in \eqref{soggebounds} (see \cite{bourgain}, \cite{soggekaknik}, \cite{ariturk}).

On the other hand, one of the notable aspects of the work of Burq-G\'erard-Tzvetkov \cite{bgtrestr} is that they showed an improvement on \eqref{restrictionbounds} when $n=2$ and $P$ is a curve with nonvanishing geodesic curvature.  Specifically, they proved that
\begin{equation}\label{curvedL2bound}
\|\Pi_\l f\|_{L^2(P)} \leq C \l^{\frac 16} \|f\|_{L^2(M)}.
\end{equation}
This was then generalized to all dimensions by Hu \cite{Hu} who obtained the same bound for any codimension 1 submanifold with negative definite scalar second fundamental form (or positive definite, depending on the choice of normal vector).  As before, these bounds also follow from an observation of Tataru \cite{tatarubdry} based on known estimates of H\"ormander \cite[25.3]{hormander4}.  The bound \eqref{curvedL2bound}  can then be interpolated with \eqref{restrictionbounds} when $q=\frac{2n}{n-1}$, $\delta = \frac{n-1}{2n}$ to get that the $\delta$ in \eqref{deltadefhyper} can be improved to
\begin{equation*}
\delta=\frac{n-1}{3}-\frac{2n-3}{3q} \qquad \text{when } 2 \leq q < \frac{2n}{n-1}.
\end{equation*}
These bounds thus speak to the concentration properties of eigenfunctions.  When $P$ is in some sense ``far away" from containing geodesic segments, eigenfunctions have less tendency to concentrate near $P$.  A work of Hassell-Tacy \cite{hasselltacy} proves bounds of this type in a semiclassical setting.

In the present work, we consider the development of the bounds \eqref{restrictionbounds} for $C^{1,\alpha}$ metrics with $0 \leq \alpha \leq 1$, allowing for Lipschitz regularity when $\alpha =0$.  As a corollary, we obtain bounds of this type (with a loss) for the Dirichlet and Neumann Laplacians on compact manifolds with boundary.  Bounds of the form \eqref{restrictionbounds} when $n=2$, $k=1$ for manifolds with concave boundaries are due to Ariturk \cite{ariturk}, provided Dirichlet conditions are imposed.  However, the presence of gliding rays when the manifold possesses a point of convexity within the boundary complicates matters considerably.

\begin{theorem}\label{thm:general}
Suppose $\g \in C^{1,\alpha}$ with $0 \leq \alpha \leq 1$, allowing for Lipschitz regularity when $\alpha =0$.  When $k=n-1$ and $2 \leq q \leq \frac{2n}{n-1}$, we have that for $\delta = \frac{n-1}{4}-\frac{n-2}{2q}$
\begin{equation}\label{lowregrestrsupercrit}
\|\Pi_\l f\|_{L^q(P)} \leq C \l^{\delta(1 + \sigma) } \|f\|_{L^2(M)}, \qquad
\sigma = \frac{1-\alpha}{3+\alpha}.
\end{equation}
Moreover, when $k=n-1$, $\frac{2n}{n-1} \leq q \leq \infty$ or $k \leq n-2$ we suppose that $\delta= \frac{n-1}{2}-\frac kq$ and $\delta + \frac{\sigma}q <1+\alpha$ with $\sigma$ as above. In this case, the following bounds are satisfied,
\begin{equation}\label{lowregrestrbounds}
\|\Pi_\l f\|_{L^q(P)} \leq C \l^{\delta + \frac{\sigma}q } \|f\|_{L^2(M)}
\end{equation}
with $C$ replaced by $C(\log\l)^{\frac 12}$ when $(k,q) = (n-2,2)$.  The admissibility condition on $\delta$, $q$ can be relaxed to $\delta + \frac{\sigma}q \leq 1+\alpha$ when $\alpha =0$ or $\alpha =1$.
\end{theorem}

Furthermore, we will show improvements akin to \eqref{curvedL2bound} when $0< \alpha \leq 1$.  For these metrics the Christoffel symbols are well defined and continuous on $M$ by the usual coordinate formula
$$
\Gamma^k_{ij} = \frac 12 \g^{kl}\left(\prtl_i \g_{jl} + \prtl_j \g_{il} - \prtl_l \g_{ij}\right)
$$
(with the summation convention in effect).  Hence there is also a well defined Levi-Civita connection associated to the metric $\g$ on $M$, mapping $C^1$ vector fields to continuous vector fields with the usual properties.  In particular, given a smooth, embedded, codimension 1 submanifold of $P$, the scalar second fundamental form is well defined and if it is negative definite throughout $P$ for a suitable choice of normal vector field, we shall call it ``curved".  We will see that in this case, the power of $\l$ in \eqref{lowregrestrsupercrit} with $q=2$ can be improved to $\frac 16 + \frac{\sigma}{2}$ (which can be seen as strictly less than $\frac 14(1+\sigma)$ when $\sigma < \frac 13$).

\begin{theorem}\label{thm:curved}
Suppose $\g \in C^{1,\alpha}$ with $0 < \alpha \leq 1$, and that $P$ is a ``curved" codimension 1 submanifold as defined above.  Then the following bounds are satisfied
\begin{equation}\label{lowregcurvedbounds}
\|\Pi_\l f\|_{L^2(P)} \leq C \l^{\frac 16+\frac{\sigma}{2}} \|f\|_{L^2(M)}, \qquad
\sigma = \frac{1-\alpha}{3+\alpha}.
\end{equation}
Moreover, interpolating this bound with the $q=\frac{2n}{n-1}$ case of \eqref{lowregrestrsupercrit} yields an improvement of that estimate for $2 \leq q <\frac{2n}{n-1}$.
\end{theorem}

Following \cite{SmC1alph}, we will show that for each theorem, the $0 \leq \alpha <1$ case follows from the  $\alpha =1$ case by rescaling methods.  This involves dilating coordinates so that sets of diameter $\approx\lambda^{-\sigma}$ in $P$ have diameter $\approx 1$ in the new coordinates.  Since the metric can be approximated by one with $C^{1,1}$ regularity here, the bounds from the $\alpha =1$ case can then be applied.  In the original coordinates, this then implies that the estimates \eqref{lowregrestrsupercrit}, \eqref{lowregrestrbounds}, \eqref{lowregcurvedbounds} hold with $\sigma = 0$ over sets of diameter $\approx\lambda^{-\sigma}$.  By incorporating the flux estimates in \cite{SmC1alph}, it can then be seen that Theorems \ref{thm:general} and \ref{thm:curved} follow by taking a sum over all such sets.

The bounds \eqref{soggebounds} for $C^{1,1}$ metrics in \cite{SmC2} (and those for manifolds with boundary in \cite{smithsogge06}) were proved by wave equation methods.  Specifically, square function estimates are developed for solutions to the wave equation on these manifolds, bounding the $L^q(M)$ norm of the square function
$$
x \mapsto \left( \int_{0}^1 |u(t,x)|^2\,dt \right)^{\frac 12}, \qquad \text{where } (\prtl^2_t -\Delta_\g)u=0.
$$
As will be seen below, spectral clusters above naturally give rise to solutions to the wave equation, these estimates imply bounds on the $\Pi_\l f$.  Square function estimates were first proved in \cite{mss93} for smooth metrics, using that Fourier integral operators can be used to invert the equation.  However, when $\g \in C^{1,1}$, the roughness of the metric means that these methods are inapplicable, so a crucial development in \cite{SmC2} is the construction of a suitable parametrix using wave packet methods.  The resulting approximate solution operators can be thought of as generalized Fourier integral operators where the associated canonical relation satisfies the curvature condition in \cite{mss93}.

We follow the same strategy here, essentially proving bounds on the $L^q(P)$ norm of the square function above.  Once again, the roughness of the metric means that we are led to use wave packet methods to construct a parametrix.  In this case, the canonical relations which arise naturally have folding singularities.  In Theorem \ref{thm:general}, the relation has a one-sided fold and in Theorem \ref{thm:curved} the relation essentially has a two-sided fold.  There is a significant body of work on $L^2 \to L^q$ bounds for Fourier integral operators with folding singularities, see \cite{greenleafseeger}, \cite{hormander4}, \cite{meltaykirchoff}, \cite{pansogge}, \cite{cuccagna} (the first one treating one sided folds).  A key technical development in the present work is that the operators arising from the wave packet transform satisfy the desired square function estimates in spite of the inapplicability of these results for Fourier integral operators.  Nonetheless, the approach taken here is in part inspired by these works.

\subsection*{Notation} We use $C^\alpha$ to H\"older class of order $\alpha$.  Moreover $C^{1,\alpha}$ will denote the class of metrics or functions whose first derivative is in $C^\alpha$, taking the contrived convention that Lipschitz regularity is allowable when $\alpha=0$.  In what follows, $X \lesssim Y$ will denote that $X \leq CY$ for some implicit constant $C$ which is in some sense uniform, though when used in decay estimates, it may depend on the order $N$.  Similarly, $X \approx Y$ will denote that $X \lesssim Y$ and $Y \lesssim X$.  We use $d$ as the differential which carries scalar functions to covector fields and vectors into matrices in the natural way.  Given a metric $\g$ under discussion, we let $\langle\cdot,\cdot\rangle_\g$, $|\cdot|_\g$ denote the inner product and length induced by the metric either in the tangent or cotangent space. Lastly, given a vector $x \in \RR^n$, $x'$ and $x''$ will typically denote a vector in $\RR^l$, $l <n$, formed by taking a subcollection of the components of $x$.  The nature of this subcollection may vary depending on the section.

\subsection*{Remark on admissibility conditions}  The admissibility condition $\delta + \frac{\sigma}q <1+\alpha$ (with equality allowed when $\alpha = 0,1$) arises in \S\ref{sec:reductions} where elliptic regularity is used to show that when a cluster $\Pi_\l f$ is considered in a coordinate system, the high frequency components (with respect to the Fourier transform) satisfy better bounds than those near frequency $\l$.  However, it can be checked that the condition $\delta < \frac 12$ is always satisfied when $k=n-1$ and $2 \leq q \leq \frac{2n}{n-1}$ and that $\delta < \frac 56$ holds for sufficiently small $q>2$ when $k=2$ ensuring that in many relevant cases, the admissibility condition is satisfied.  On the other hand, Smith \cite[p. 969]{SmC1alph} showed that the bound $\|\Pi_\l f\|_{L^\infty(M)} \lesssim \l^{\frac{n-1}{2}} \|f\|_{L^2(M)}$ holds whenever $\g$ is Lipschitz.  The key observation here is that one can write $\Pi_\l f = \exp(-\l^{-2}\Delta_\g) \Pi_\l \tilde{f}$ with $\|\Pi_\l \tilde{f}\|_{L^2(M)} \approx \|\Pi_\l f\|_{L^2(M)}$.  The $L^\infty(M)$ bounds then follow by combining Saloff-Coste's \cite{saloffcoste} Gaussian upper bounds on the heat kernel with Smith's $L^{\frac{2(n+1)}{n-1}}(M)$ bounds on $\Pi_\l f$.  However, the same argument gives the continuity of each $\Pi_\l f \in L^2(M)$ since the fixed time heat kernel is continuous on $M\times M$ (as observed in \cite[\S6]{saloffcoste}).  Thus Smith's $L^\infty$ bounds on spectral clusters imply $L^\infty$ bounds on their restrictions and this can be interpolated with the $L^q(P)$ bounds for submanifolds of low codimension to see that in many cases, the admissibility conditions can be relaxed.  This also ensures that the restrictions are well-defined.

\subsection*{Remark on the optimality of \eqref{lowregrestrsupercrit}}
As noted above in \eqref{lowregexample}, the examples in \cite{sS94}, \cite{smithtatexample} show that the bounds of Smith \cite{SmC1alph} establishing $L^q(M)$ bounds are sharp for small values of $q>2$.  We comment here that the same examples show that the bounds \eqref{lowregrestrsupercrit} in Theorem \ref{thm:general} are sharp as well.  Indeed, the examples in \cite{sS94} produce metrics of $C^{1,\alpha}$ regularity and associated spectral clusters $f_\l$ which are concentrated in a tube of length 1 and diameter $\l^{-\frac{2}{3+\alpha}}$, that is, a set of the form
\begin{equation}\label{alphatube}
|x_1| \lesssim 1 \qquad |(x_2,\dots, x_n)| \lesssim \l^{-\frac{2}{3+\alpha}}.
\end{equation}
Therefore if we take $P$ to be defined by $x_n=0$, we see that the rapid decay outside of this set implies that
$$
\frac{\|f_\l\|_{L^q(P)}}{\|f_\l\|_{L^2(M)}} \approx \l^{\frac{2}{3+\alpha}(\frac{n-1}{2}-\frac{n-2}{q})}.
$$
However, $\frac 12(\sigma+1) = \frac{2}{3+\alpha}$, showing that the exponent simplifies to $\delta(1+\sigma)$ and hence the bound \eqref{lowregrestrsupercrit} is optimal.

\subsection*{Acknowledgements}
The author wishes to thank the anonymous referee for pointing out a significant error in an earlier draft of the paper and other helpful suggestions.

\section{Microlocal reductions}\label{sec:reductions}
In this section, we will reduce the main theorems to proving square function estimates for frequency localized solutions to a hyperbolic pseudodifferential equation.  We follow an approach due to H. Smith \cite{SmC1alph} (see also \cite{bssmulti}).  The needed reductions are fairly common to both theorems, so we begin by treating all cases at the same time.  It is thus convenient to take the convention that $\delta(\sigma)$ is defined by taking the power of $\l$ appearing in \eqref{lowregrestrsupercrit}, \eqref{lowregrestrbounds}, or \eqref{lowregcurvedbounds}, realizing that in all cases $\delta(0)$ denotes the power without loss of derivatives.  Moreover, the admissibility conditions mean that if $\sigma >0$ and $0 < \alpha < 1$, $\delta(\sigma)-1 <\alpha$ (respectively $\delta(\sigma)-1\leq \alpha$ when $\alpha =0,1$).

Throughout these preliminary reductions, we will make use of the fact that when $k < n$, we have the following embedding for traces in $\RR^k \times \{0\}$, $\{0\} \in \RR^{n-k}$
\begin{equation}\label{sobtrace}
H^{\frac n2 - \frac kq}(\RR^n) \hookrightarrow L^q(\RR^k \times \{0\}),
\end{equation}
which can be seen by first applying Sobolev embedding on $\RR^k \times \{0\}$, then using the trace theorem for $L^2$ based Sobolev spaces.  The estimates in Theorems \ref{thm:general} and \ref{thm:curved} thus exhibit a gain relative to Sobolev embedding.  The gain is largest when $q=2$, in which case there is a gain of a 1/4 or 1/3 of a derivative when $k=n-1$ depending on whether the submanifold is curved and a gain of 1/2 a derivative (up to a possible logarthmic correction) when $k \leq n-2$.

It suffices to prove the main theorem for a spectral cluster $f$ satisfying $f = \Pi_\l f.$
We begin by observing that $f$ satisfies the following bounds in Sobolev spaces defined by the spectral resolution of $\Delta_\g$
\begin{equation*}
\|(\Delta_\g + \l^2) f\|_{H^s(M)} + \|d f\|_{H^s(M)} \lesssim \l^{s+1} \|f\|_{L^2(M)}.
\end{equation*}
It thus suffices to prove bounds on $f$ of the following form
\begin{equation}\label{mfldbound}
\|f\|_{L^q(P)} \lesssim \sum_i \l^{\delta(\sigma)-1-s_i}\left(\l\|f\|_{H^{s_i}(M)} + \| d f \|_{H^{s_i}(M)} +
\|(\Delta_\g + \l^2) f\|_{H^{s_i}(M)}\right)
\end{equation}
where a sum is taken over a finite collection of $0\leq s_i \leq 1$.

Multiplication by any smooth bump function $\psi$ preserves $H^1(M)$, and by interpolation $H^s(M)$ for any $s \in [0,1]$.  Therefore by taking a partition of unity on $M$, it suffices to prove \eqref{mfldbound} with $f$ replaced by $\psi f$, where $\psi$ is supported in a suitable coordinate chart which intersects $P$.  Specifically, we will take slice coordinates so that $P$ is identified with $\RR^k \times \{0\}$.  Furthermore, by taking a sufficiently fine partition of unity and dilating coordinates, we may assume that for some $c_0$ sufficiently small,
\begin{equation}\label{c0first}
\|\g^{ij}-\delta_{ij}\|_{C^{1,\alpha}(\RR^n)} \leq c_0.
\end{equation}

By elliptic regularity (see e.g. Gilbarg-Trudinger \cite[Theorem 8.10, Theorem 9.11]{GT}) and interpolation we have that for any $g$ supported in this coordinate chart
$
\|g\|_{H^s(M)} \approx \|g\|_{H^s(\RR^n)}$ for $s \in [0,2].
$
Next we observe that in coordinates within $\supp(\psi)$, $f$ satisfies an equation of the form
\begin{equation}\label{quasimodeRn}
\g \dsq f + \l^2 f = w, \qquad \g \dsq f = \sum_{1 \leq i,j \leq n} \g^{ij}\prtl^2_{ij} f
\end{equation}
where $w$ is a sum consisting of $(\Delta_\g +\l^2)f$ and products of the form $a \cdot \prtl_j f$, with $a \in C^\alpha$ (or $L^\infty$, $C^{0,1}$ when $\alpha =0,1$ respectively) in turn a product of functions of the form $\g^{ij}$, $\sqrt{\det{\g_{ij}}}$ or their first derivatives.  Hence multiplication by these functions preserves $H^s(\RR^n)$ for $s=0$ and $s \in [0,\alpha)$ when $\alpha >0$ (respectively $s \in [0,1]$ when $\alpha =1$) meaning that that for any such $s$
$$
\|w\|_{H^{s}(\RR^n)} \lesssim \|f\|_{H^{s}(M)}+\|d f\|_{H^{s}(M)}+ \|(\Delta_\g + \l^2) f\|_{H^{s}(M)}.
$$

Furthermore, elliptic regularity (see e.g. \cite[Theorem 9.11]{GT}) also gives that
\begin{equation}\label{ellipticfirst}
\|\dsq f\|_{L^2(\RR^n)} \lesssim \|\Delta_\g f\|_{L^2(\RR^n)}+\|df\|_{L^2(\RR^n)} + \|f\|_{L^2(\RR^n)}.
\end{equation}
Moreover, when $\delta(\sigma) > 1$ (which only occurs when $\alpha >0$), we have that
\begin{equation}\label{commuteDs}
\|[\g^{ij},\langle D\rangle^{\delta(\sigma) -1}]\prtl^2_{ij} f \|_{L^2(\RR^n)} + \|[\prtl_i \g^{ij} ,\langle D\rangle^{\delta(\sigma) -1}] \prtl_j f\|_{L^2(\RR^n)} \lesssim \|df\|_{H^{{\delta(\sigma) -1}}(\RR^n)},
\end{equation}
where $\langle D \rangle$ denotes the Fourier multiplier with symbol $(1+|\xi|^2)^{\frac 12}$.  This means that we may replace $L^2$ by $H^{\delta(\sigma)-1}$ in \eqref{ellipticfirst}.  Indeed, the bound on the first term in \eqref{commuteDs} follows as a consequence of the Coifman-Meyer commutator theorem (see e.g. \cite[Proposition 3.6B]{tay91}) and the second follows since the admissibility condition on $\delta(\sigma)$ implies that multiplication by $\prtl_i \g^{ij}$ preserves $H^{\delta(\sigma)-1}(\RR^n)$.

With this in mind, we define the following norm when $\delta(\sigma) \leq 1$
$$
|||f||| := \|f\|_{L^2(\RR^n)} + \l^{-1}\| df \|_{L^2(\RR^n)} + \l^{-2}\| \dsq f \|_{L^2(\RR^n)} +
\l^{-1}\|w\|_{L^2(\RR^n)}.
$$
When $\delta(\sigma) >1$, we define
\begin{multline*}
|||f||| := \sum_{j=0 }^2\l^{-j}\|d^j\! f\|_{L^2(\RR^n)} + \l^{-1}\|w\|_{L^2(\RR^n)}\\
+ \l^{-(\delta(\sigma)-1)}\left(\sum_{j=0 }^2\l^{-j}\|d^j\! f\|_{H^{\delta(\sigma)-1}(\RR^n)} +\l^{-1}\|w\|_{H^{\delta(\sigma)-1}(\RR^n)}\right).
\end{multline*}
Given the observations above, it now suffices to show that
\begin{equation}\label{tripnormbound}
\|f\|_{L^q(\RR^k \times \{0\})} \lesssim \l^{\delta(\sigma)}|||f|||.
\end{equation}

Without loss of generality, we may assume that $f$ is supported in a cube of sidelength 1 centered at the origin and that the metric is defined over a cube of sidelength 8 centered at the origin.  Hence we may smoothly extend the metric $\g$ so that it is defined over all of $\RR^n$ and equal to the flat metric for $|x|$ sufficiently large without altering the equation for $f$. Given $r>0$, we let $S_r = S_r(D)$ denote a Fourier multiplier which applies a smooth cutoff to frequencies $|\xi| \leq r$ and define $\g_\l = S_{c^2 \l}\g$ where $c>0$ will be taken to be sufficiently small.  Since
\begin{equation}\label{lambdaapprox}
\|\g_\l -\g\|_{L^\infty} \lesssim \l^{-1}
\end{equation}
we may replace $\g$ by $\g_\l$ in \eqref{quasimodeRn} when $\delta(\sigma)\leq 1$ as the error can be absorbed into the right hand side of \eqref{tripnormbound}.  The same holds when $1<\delta(\sigma)$ is admissible, which can be seen by using the similar bound $\|\g_\l -\g\|_{C^\alpha} \lesssim \l^{-1}$ and the fact that multiplication by a $C^\alpha$ function preserves $H^{\delta(\sigma)-1}(\RR^n)$.

We now write $f$ as $f= f_{<\l} + f_{\l}+f_{>\l}$ where $f_{<\l} = S_{c\l}f$ and $f_{>\l} = f-S_{c^{-1}\l}f$.  Observe that when $s =0 $
\begin{equation}\label{glcomm}
\|[S_{c\l},\g_\l]\|_{H^s \to H^s} +\|[S_{c^{-1}\l},\g_\l]\|_{H^s \to H^s}\lesssim \l^{-1}
\end{equation}
which follows from simple bounds on the kernel of the commutators.  When $1< \delta(\sigma)$ is admissible, the same holds with $s=\delta(\sigma)-1$.   Indeed, we have that $\l S_{c\l}$ (and similarly $\l S_{c^{-1}\l}$) defines an operator in $S^1_{1,0}$ hence the symbolic calculus gives $[\l S_{c\l},\g_\l] \in C^\alpha S^0_{1,0}$ (in the notation of \cite{tay91}).  The claim then follows by \cite[Proposition 2.1D]{tay91} or by commuting with derivatives when $\alpha =1$.  Defining $ w_{<\l}:= \g_\l \dsq f_{<\l} + \l^2 f_{<\l}$, $w_{>\l}:= \g_\l \dsq f_{>\l} + \l^2 f_{>\l}$
we have
\begin{equation}\label{lowfreqsource}
\|w_{<\l}\|_{H^s(\RR^n)} + \|w_{>\l}\|_{H^s(\RR^n)} \lesssim \l^{-1}\|\dsq f \|_{H^s(\RR^n)} + \|w\|_{H^s(\RR^n)}
\end{equation}
for $s=0$ and for $s = \delta(\sigma)-1$ when the latter quantity is positive.

To bound $f_{<\l}$, $f_{>\l}$ we use arguments from \cite[Corollary 5]{SmC1alph}.  Since $\|\g_\l \dsq f_{<\l}\|_{L^2} \lesssim (c\l)^2 \|f_{<\l}\|_{L^2} $, \eqref{sobtrace} and the equation give the stronger estimate
$$
\|f_{<\l}\|_{L^q(\RR^k \times \{0\})} \lesssim \l^{\frac n2 -\frac kq} \|f_{<\l}\|_{L^2(\RR^n)} \lesssim \l^{\frac n2 -\frac kq -2} \|w_{<\l}\|_{L^2(\RR^n)} \lesssim \l^{\frac n2 -\frac kq -1} |||f|||.
$$
For the high frequency term $f_{>\l}$, we use that when $s \geq 0$,
$$
\l^2\|f_{>\l}\|_{H^s(\RR^n)} + \l\|df_{>\l}\|_{H^s(\RR^n)} \lesssim c\|\dsq f_{>\l}\|_{H^s(\RR^n)}.
$$
This bound with $s=0$ can be combined with elliptic regularity to obtain
\begin{equation}\label{elliptichifreq}
\|\dsq f_{>\l}\|_{L^2(\RR^n)} \lesssim \|w_{>\l}\|_{L^2(\RR^n)}.
\end{equation}
When $\frac n2 -\frac kq \leq 2$, \eqref{sobtrace} yields a gain of at least 1/2 of a derivative in the estimate for $f_{>\l}$.  The case $\frac n2 -\frac kq >2$ only arises when $\alpha >0$ and $\delta(\sigma) = \frac{n-1}{2}-\frac kq+\frac{\sigma}{q}$, and in this case we use \eqref{commuteDs} (with $\g_\l$ replacing $\g$) to  bootstrap the elliptic regularity estimate, which yields a similar gain for $f_{>\l}$ since
$$
\| f_{>\l}\|_{H^{\delta(\sigma)+1}(\RR^n)} \lesssim \|w_{>\l}\|_{H^{\delta(\sigma)-1}(\RR^n)} \lesssim |||f|||.
$$

We are now reduced to proving bounds on $f_{\l}$.  Reasoning as in \eqref{glcomm}, we have that $|||f_\l||| \lesssim |||f|||$.  We now impose a further microlocal decomposition of the function, writing $f_\l = f_{\l,T}+f_{\l,N}$, where $\widehat{f}_{\l,T}$ is localized to directions tangent to the submanifold and $\widehat{f}_{\l,N}$ is localized to normal directions.  Specifically, we write $f_{\l,N} = \sum_{j=k + 1}^n f_{\l,j}$ where $f_{\l,j}$ is frequency localized to a set of the form
\begin{equation*}
\supp(\widehat{ f_{\l,j} }) \subset \{\xi:\l \approx |\xi|, \,|\xi_j| \gtrsim \veps |(\xi_1,\dots,\xi_j,\xi_{j+1}, \dots, \xi_n)| \},
\end{equation*}
with $\veps$ suitably small.  Using \eqref{glcomm} again, we have that
\begin{equation}\label{gammacomm}
\|\g_\l \dsq f_{\l,j} + \l^2 f_{\l,j}\|_{L^2(\RR^n)} \lesssim |||f_\l||| .
\end{equation}
With this in mind, the flux estimates of Smith \cite[p.974]{SmC1alph}, give
\begin{equation}\label{flux}
\|f_{\l,j}\|_{L^\infty_{x_j}L^2_{x'}} \lesssim |||f_\l|||
\end{equation}
where $x'$ denotes the vector consisting of every component in $\RR^n$ but $x_j$.  Combining this with the $n-1$ dimensional version of \eqref{sobtrace} on the hyperplane $x_j=0$, we have
$$
\|f_{\l,j}\|_{L^q(\RR^k \times \{0\}) } \lesssim \l^{\frac{n-1}{2}-\frac kq}\|f_{\l,j}\|_{L^2{(x_j=0)}} \lesssim \l^{\delta(\sigma)}|||f_\l|||.
$$

We now further decompose $f_{\l,T}$ as $f_{\l,T} = \sum_{j} f_{\l,\omega_j}$ where $\{ \omega_j\}$ is a finite collection of unit vectors and $\supp(\widehat{ f_{\l,\omega_j} }) $ lies in a small conic set containing $\omega_j$.  Without loss of generality, it suffices to treat the case $\omega_j = -e_1 = (-1,0,\dots,0)$.  Recalling \eqref{gammacomm} and simplifying notation it now suffices to prove $\|f_\l\|_{L^q(\RR^k \times \{0\})} \lesssim \l^{\delta(\sigma)} |||f_\l|||$ for $f_\l$ satisfying
\begin{equation}\label{conelast}
\supp(\widehat{ f_{\l} }) \subset \{\xi:\left|\xi/|\xi| - (-e_1) \right| \lesssim \veps\}.
\end{equation}

As a consequence of \eqref{flux} with $x_j = x_1$ and H\"older's inequality, we have that if $S_R$ is a slab of the form $S_R = \{x: |x_1-r| \leq R \}$ for some $r$
\begin{equation}\label{fluxR}
\|f_{\l}\|_{L^2(S_R)} \lesssim R^{\frac 12}|||f_\l|||.
\end{equation}
Set $\rho = \frac{n-1}{2}-\frac kq$ so that $\rho-\delta(0) = \delta(0) - \frac 1q$ when $k=n-1$, $2 \leq q \leq \frac{2n}{n-1}$ and  $\rho=\delta(0)$ in all other cases of Theorem \ref{thm:general}.  Given a cube $Q_R$ of sidelength $R = \l^{-\sigma}$ which intersects $\RR^k \times \{0\}$, we let $Q_R^*$ denote its double, and also set $w_\l := \g_\l \dsq f_\l + \l^2 f_\l$.  We claim that Theorem \ref{thm:general} now follows from the bound
\begin{equation}\label{qrestimate}
\|f_\l\|_{L^q\left((\RR^k \times \{0\}) \cap Q_R\right)}  \lesssim \l^{(1-\sigma)\delta(0)} R^{-\rho}\left( R^{-\frac 12} \|f_{\l}\|_{L^2(Q_R^*)} + R^{\frac 12} \l^{-1}\|w_{\l}\|_{L^2(Q_R^*)} \right).
\end{equation}
Moreover, Theorem \ref{thm:curved} will follow from taking $q=2$, $\delta(0) = 1/6$ here when $P$ is curved (as $\rho=0$ in this case).  Indeed, if these bounds hold, we may sum over the cubes $Q_R$ contained in $S_R$ which intersect $\RR^k \times \{0\}$ to obtain
$$
\|f_\l\|_{L^q\left((\RR^k \times \{0\}) \cap S_R\right)}  \lesssim \l^{(1-\sigma)\delta(0)+\sigma \rho} \left( R^{-\frac 12} \|f_{\l}\|_{L^2(S_R^*)} + R^{\frac 12} \l^{-1}\|w_{\l}\|_{L^2(S_R^*)} \right) .
$$
Recalling \eqref{fluxR}, the right hand side is bounded by $\l^{(1-\sigma)\delta(0)+\sigma \rho} |||f_\l|||$.  Given the previous observations on $\rho$, the desired bound on $f_\l$ then follows by taking a sum over the $\mathcal{O}(R^{-1})$ slabs $S_R$ in $|x_1| \leq 3/4$ and the rapid decay property
\begin{equation}\label{rapiddkf}
|f_\l(x)| \lesssim (\l |x|)^{-N} \|f_\l \|_{L^2(\RR^n)}\qquad \text{for } \max_{j}|x_j| \geq 3/4.
\end{equation}
The latter is a consequence of our assumption that $f$ is supported in a cube of sidelength 1 at the origin, which implies that $f_\l$ is concentrated in a $\l^{-1}$ neighborhood of this cube.

At this stage, we pause to remark on a useful feature of our metric when $P$ is curved.  Let $N$ be a suitable unit normal vector field such that $\langle N, \prtl_n\rangle >0$. Observe that given any $n-1$ vector $(X^1, \dots, X^{n-1})$ such that $(X^1)^2 + \cdots + (X^{n-1})^2=1$, we may assume that over $P$, the quantity
\begin{equation}\label{negdefuniform}
-\sum_{1 \leq i,j \leq n-1}\langle N, \nabla_{\prtl_i}\prtl_j\rangle_{\g} X^i X^j
\end{equation}
is uniformly bounded from above and below.  Indeed, since $\prtl_1, \dots, \prtl_{n-1}$ span the tangent space to $P$ one just applies the assumption that $P$ is curved to constant vector fields of the form $X^j\prtl_j$ (with summation convention in effect). Using that $\nabla_{\prtl_i}\prtl_j$ is the vector field $\Gamma^k_{ij}\prtl_k$, we may use that $\langle N, \prtl_k\rangle_{\g}\equiv 0$ on $P$ for $k \neq n$ and that $\langle N, \prtl_n\rangle_{\g}$ is bounded above to get that
\begin{equation}\label{negdefglambda}
-\sum_{1 \leq i,j \leq n-1} \Gamma^n_{ij}X^iX^j
\end{equation}
is uniformly bounded from above and below over $P$ for all such $(X^1, \dots, X^{n-1})$.  Using that $\|\g - \g_{\l}\|_{C^1} \lesssim \l^{-\alpha}$, the bounds also hold when the Christoffel symbols are taken with respect to $\g_\l$.

Returning to the proof of \eqref{qrestimate}, we dilate variables $x \mapsto Rx$, set $\mu := R\l$, and make the slight abuse of notation that $f_\mu(x) = f_\l(Rx)$.  We will see that this reduces the general bounds to those without a loss of derivatives, and hence we will take $\delta= \delta(0)$ below.  Indeed, rescaling the bound \eqref{qrestimate} gives
\begin{equation}\label{Qestimate}
\|f_\mu\|_{L^q\left((\RR^k \times \{0\}) \cap Q\right)}  \lesssim \mu^\delta \left( \|f_\mu\|_{L^2(Q^*)} + \mu^{-1}\|\g_\mu \dsq f_\mu + \mu^2 f_\mu \|_{L^2(Q^*)} \right).
\end{equation}
When $P$ is curved, rescaling yields the same with $q=2$ and $\delta = (1+\beta)/6$ where $\beta = \sigma/(1-\sigma)$.  Here $Q$ is now a cube of sidelength 1, which we may take to be centered at the origin, and $\g_\mu(x):=\g_\l(R x)$.  We now have that if $\g_{\mu^{1/2}}:=S_{c^2\mu^{1/2}}\g_\mu$, then (cf. \eqref{c0first})
\begin{equation}\label{c2approx}
\|\g_\mu-\g_{\mu^{1/2}}\|_{L^\infty} \lesssim c_0\mu^{-1}
\end{equation}
and we may replace $\g_\mu$ by $\g_{\mu^{1/2}}$ in \eqref{Qestimate} since the error can be absorbed in to the right hand side.  The metric $\g_{\mu^{1/2}} $ has $C^2$ regularity, namely
\begin{equation}\label{c0}
\|\g_{\mu^{1/2}}^{ij} -\delta_{ij}\|_{C^2} \lesssim c_0 \quad \text{and} \quad \|\prtl^{\alpha}\g_{\mu^{1/2}}^{ij}\|_{C^2} \leq \mu^{\frac 12(|\alpha|-2)} \text{ for } |\alpha| \geq 2.
\end{equation}

We pause again to discuss the effect of this dilation and regularization on the upper and lower bounds on \eqref{negdefglambda} for curved metrics.  For unit $n-1$ vectors $(X^1,\dots,X^{n-1})$, we now have
\begin{equation}\label{negdefgmu}
c_1 \leq -\mu^{\beta} \sum_{1 \leq i,j \leq n-1} \Gamma^n_{ij}(x)X^iX^j \lesssim c_0
\end{equation}
for $x \in P$.  Here the Christoffel symbols can be taken with respect to the metric $\g_{\mu^{1/2}}$ since we now have \eqref{c2approx} and $\|\g_\mu-\g_{\mu^{1/2}}\|_{C^{1,\alpha}} \lesssim \mu^{-\frac 12}\ll \mu^{-\beta}$.  Moreover, by continuity we may assume that if $c_0$ is chosen sufficiently small, then the inequality holds for all $x \in Q$ at the expense of decreasing $c_1$ slightly.

We will prove the bound \eqref{Qestimate} by wave equation methods.  Let $u_\mu(t,x) = \cos(t\mu) f_\mu(x)$.  It suffices to show that if $F_\mu=(\prtl^2_t - \g_{\mu^{1/2}} d^2) u_\mu$
$$
\|u_\mu\|_{L^q((\RR^k \times \{0\}) \cap Q;L^2(-\frac 12,\frac 12))} \lesssim \mu^{\delta}\left(\|u_\mu(0,\cdot)\|_{L^2(Q^*)}+
\mu^{-1}\|F_\mu\|_{L^2((-1,1)\times Q^*)} \right).
$$
Now let $\psi(t,x)$ denote a smooth cutoff identically one on $(-\frac 12,\frac 12)^{n+1}$ and supported in  $(-\frac 34,\frac 34)^{n+1}$.  Replacing $u_\mu$ by $\psi u_\mu$, and similarly for $F_\mu$, it suffices to show that
\begin{equation}\label{wave2ndorder}
\|u_\mu\|_{L^q(\RR^{k} \times \{0\};L^2(\RR))} \lesssim \mu^{\delta}\left(\|u_\mu(0,\cdot)\|_{L^2(\RR^n)}+
\mu^{-1}\|F_\mu\|_{L^2(\RR^{n+1})} \right)
\end{equation}
since energy estimates bound the error terms which arise when commuting $(\prtl^2_t - \g_{\mu^{1/2}} d^2)$ with $\psi$.  Next we let $\Gamma_\mu^\pm(\tau,\xi)$ be smooth cutoffs to regions of the form
\begin{equation}\label{pmregions}
\{(\tau, \xi): \pm \tau \approx |\xi|, |\xi| \approx \mu, |\xi/|\xi| - (-e_1)| \lesssim \veps \}
\end{equation}
and supported in a slightly larger set.  Let $u_\mu^\pm = \Gamma_\mu^\pm(D_{t,x}) u_\mu$.  By \cite[Lemma 2.3]{SmC2} and the localization of $f_\mu$, we see that elliptic regularity and \eqref{sobtrace} yields an estimate on $u_\mu-u_\mu^+ -u_\mu^-$ with a gain of at least a half a derivative relative to the right hand side of \eqref{wave2ndorder}.  It thus suffices to prove \eqref{wave2ndorder} with $u_\mu$ replaced by $u_\mu^\pm$.  The proof of the bound will follow in the next two sections.

\section{General Submanifolds}\label{sec:general}
In this section, we prove \eqref{wave1storder} and hence Theorem \ref{thm:general}. Recall that coordinates are chosen so that $P$ is identified with $(y,0) \in \RR^n$ with $y \in \RR^k$, $0 \in \RR^{n-k}$.  In this section, we take the following notational conventions on coordinates in $\RR^n$.  The letters $w,y,z$ will denote vectors in $\RR^k$, and given such a vector we let $\bar{y}$ denote the vector in $\RR^n$ determined by $\bar{y} = (y,0)$.  The letters $x,\xi,v$ will typically denote vectors in $\RR^n$ and we will often decompose such a vector as $x=(x_1,x',x'')$ where $x'=(x_2,\dots,x_k)$, $x''=(x_{k+1},\dots,x_{n})$.

We begin by showing that $u_\mu^\pm$ solves an equation which is hyperbolic in $x_1$.  Given \eqref{c0}, we have that for $(\tau,\xi)$ in the regions \eqref{pmregions}, $\g^{ij}_{\mu^{1/2}}\xi_i\xi_j-\tau^2$ defines a quadratic in $\xi_1$ with two real roots and hence we may write
\begin{equation}\label{quadratic}
\g^{ij}_{\mu^{1/2}}(x)\xi_i\xi_j -\tau^2 = \g^{11}_{\mu^{1/2}}(x)\left(\xi_1 + q^-(x,\tau,\xi')\right) \left(\xi_1 - q^+(x,\tau,\xi')\right)
\end{equation}
with $q^\pm >0$ and homogeneous of degree 1 for such $(\tau,\xi)$.  We further regularize these symbols taking $p^\pm(\cdot, \tau,\xi') = S_{c^2\mu^{\frac 12}} q^\pm(\cdot, \tau,\xi')$.  By the elliptic regularity argument in \cite[Lemma 2.4]{SmC2}, the function $u_\mu$  satisfies
\begin{equation}\label{equation1storder}
\left(- i\prtl_{x_1}  + p^\pm(x,D_{t,x'})\right) u_\mu^\pm = G_\mu^\pm,
\end{equation}
with $\|G_\mu^\pm\|_{L^2(\RR^{n+1})}$ bounded by the terms in parentheses on the right hand side of \eqref{wave2ndorder}.  Moreover, akin to \eqref{rapiddkf}, we have the rapid decay property
\begin{equation}\label{rapiddku}
|u_\mu^\pm(t,x)| \lesssim (\mu |(t,x)|)^{-N} \|u_\mu\|_{L^2(\RR^{n+1})}, \quad \text{for } \max(|t|,|x_1|,\dots,|x_n|) \geq 1.
\end{equation}
Thus by energy estimates it can be seen that
$$
\|u_\mu^\pm\|_{L^2(\RR^{n+1})}\lesssim \|u_\mu(0,\cdot)\|_{L^2(\RR^n)} + \mu^{-1}\|\prtl_t u_\mu(0,\cdot)\|_{L^2(\RR^n)} + \mu^{-1}\|G_\mu\|_{L^2(\RR^{n+1})}
$$
since the right hand side is compactly supported.  By \eqref{rapiddku}, it suffices to show that
\begin{equation}\label{wave1storder}
\|u_\mu^\pm\|_{L^q((-1,1) \times \RR^{k-1} \times \{0\};L^2(\RR))} \lesssim \mu^{\delta}\left(\|u_\mu^\pm\|_{L^2(\RR^{n+1})}+
\mu^{-1}\|G_\mu^\pm\|_{L^2(\RR^{n+1})} \right).
\end{equation}

It suffices to treat the term $u^-_\mu$ as bounds on the $u^+_\mu$ will follow from time reversal.  Hence we suppress the superscripts on $u_\mu^-$, $G_\mu^-$, $p^-$ below and assume the minus sign is taken when referencing \eqref{pmregions}.

It is convenient to change the roles of $t$ and $x_1$ above, and correspondingly $\tau$ and $\xi_1$, treating \eqref{equation1storder} as an equation which is hyperbolic in $t$, rather than in $x_1$.  As a consequence of \eqref{c0}, $p$ is now a function of $(t,x,\xi)$ (or more precisely $(t,x',x'',\xi)$) satisfying the bounds
\begin{equation}\label{papprox}
\left|\prtl^\gamma_{x,t}\prtl^\beta_\xi \left( p(t,x,\xi)-\sqrt{\xi_1^2-|(\xi',\xi'')|^2}\right)\right| \lesssim c_0, \qquad |\gamma|\leq 2,
\end{equation}
for $|\xi|=1$ in a cone of the form
\begin{equation}\label{epscone}
\{\xi:-\xi_1\gtrsim \veps^{-1}|(\xi',\xi'')|\}
\end{equation}
and $c_0$ can be replaced by $c_0\mu^{-\beta}$ when $|\gamma|=1$.  Moreover, for $\xi$ in the same set
\begin{equation}\label{pbounds}
\left|\prtl^\gamma_{x,t}p(t,x,\xi)\right| \lesssim \mu^{\frac 12(|\gamma|-2)}, \qquad |\gamma|\geq 2.
\end{equation}
By \eqref{rapiddku} and time translation, it suffices to prove that over the time interval $(0,1)$,
$$
\|u_\mu \|_{L^q_{t,y'}L^2_{y_1}} \lesssim \mu^{\delta}\left(\|u_\mu\|_{L^\infty_t L^2_x} + \|G_\mu\|_{L^2_{t,x}}\right)
$$
where we understand the left hand side to be
$$
\left( \int_0^1 \int_{\RR^{k-1} \times \{0\}} \left(\int_{\RR} |u_\mu(t,\bar{y}) |^2\,dy_1 \right)^{\frac q2}\,dy'\,dt\right)^{\frac 1q}, \qquad y'=(y_2,\dots,y_k)
$$
and the $L^\infty_t L^2_x$ norm on right hand side as $L^\infty((0,1);L^2(\RR^n))$.  Moreover, since $p(t,x,D)-p^*(t,x,D) \in OPS^0_{1,\frac 12}$, we may differentiate $\|u_\mu(t,\cdot)\|_{L^2_y}^2$ in $t$ to obtain
$$
\|u_\mu\|_{L^\infty_t L^2_y} \lesssim \|u_\mu\|_{L^2(\RR^{n+1})} +\|G_\mu\|_{L^2(\RR^{n+1})}.
$$

Let the wave packet transform $T_\mu: \mathcal{S}'(\RR^n) \to C^\infty (\RR^{2n})$ be defined by
$$
T_\mu f(x,\xi) = \mu^{\frac n4}\int e^{-i\langle \xi, v-x\rangle} \phi(\mu^{\frac 12}(v-x)) f(v)\,dv
$$
where $\phi$ is a real valued, radial Schwartz function such that $\supp(\widehat{\phi}\,) $ is contained in the unit ball and normalized so that $\|\phi\|_{L^2}=(2\pi)^{-\frac n2}$.  The normalization ensures that $T_\mu^* T_\mu $ is the identity on $L^2(\RR^n)$ and hence
$
\|T_\mu f\|_{L^2(\RR^{2n}_{x,\xi})} = \|f\|_{L^2(\RR^n_z)}.
$
Let $g_\mu(x):=u_\mu(0,x)$ and $\Theta_{r,t}(x,\xi)$ denote the time $r$ value of the integral curve of determined by the Hamiltonian flow of $p$ with $\Theta_{r,t}(x,\xi)|_{r=t} = (x,\xi)$. Given \cite[Lemma 3.2, Lemma 3.3]{SmC2}, we may write
\begin{equation}\label{duhamel}
(T_\mu u_\mu)(t,x,\xi) = T_\mu g_\mu(\Theta_{0,t}(x,\xi)) + \int_0^t \tilde{G}_\mu(r,\Theta_{r,t}(x,\xi))\,dr
\end{equation}
where $\tilde{G}$ satisfies
\begin{equation}\label{gtilde}
\int_0^t\|\tilde{G}_\mu(r,\cdot)\|_{L^2(\RR^{2n}_{x,\xi})}\,dr \lesssim \|u_\mu\|_{L^\infty_t L^2_v} + \int_0^t\|G_\mu(r,\cdot)\|_{L^2(\RR^{n}_{v})}\,dr,
\end{equation}
for $t \in (0,1)$.  Indeed, these lemmas show that if $H_p$ denotes the Hamiltonian vector field defined by $p$ then $T_\mu p(\cdot, D) - H_p T_\mu$ defines an operator which is bounded on $L^2$ and that \eqref{duhamel} follows by solving the corresponding transport equation.  Furthermore, given the frequency localization of  $p(\cdot,\xi)$ and the compact support of $\phi$, we may assume that uniformly in $r$, $x$, we have
\begin{equation}\label{wpsupport}
\supp((T_\mu g_\mu)(x,\cdot)), \;\supp(\tilde{G}(r,x,\cdot)) \subset \{\xi: |\xi|\approx \mu, -\xi_1 \gtrsim \veps^{-1}|(\xi',\xi'')|\}.
\end{equation}

Define the propagator
$$
W\tilde{f}(t,y) = T^*_\mu (\tilde{f}\circ \Theta_{0,t})(\bar{y}),
$$
and observe that given \eqref{duhamel}, \eqref{gtilde} it suffices to show that
\begin{equation}\label{wbound}
\|W\tilde{f}\|_{L^q_{t,y'}L^2_{y_1}} \lesssim \mu^{\delta}\|\tilde{f}\|_{L^2_{x,\xi}}
\end{equation}
with a $(\log \mu)^{\frac 12}$ loss when $(k,q)=(n-2,2)$.  Let $W_t$ denote the restricted operator $W_t\tilde{f}(y) = W\tilde{f}(r,y)|_{r=t}$.  By duality, it suffices to see that for functions $F(s,z)$
\begin{equation}\label{dualitybound}
\|WW^*F\|_{L^q_{t,y'}L^2_{y_1}} \lesssim \mu^{2\delta}\|F\|_{L^{q'}_{s,z'}L^2_{z_1}}.
\end{equation}
To prove this, we will show that
\begin{align}
\|W_tW_s^*h\|_{L^\infty_{y'}L^2_{y_1}} &\lesssim \mu^{n-1}(1+\mu|t-s|)^{-\frac{n-1}{2}}\|h\|_{L^1_{y'}L^2_{y_1}}
\label{L1Linf}
\\
\|W_tW_s^*h\|_{L^2_{y}} &\lesssim \mu^{n-k}(1+\mu|t-s|)^{-\frac{n-k}{2}}\|h\|_{L^2_{y}}\label{L2L2}
\end{align}
When $k=n-2$ and $q=2$, Young's inequality and \eqref{L2L2} give \eqref{dualitybound} with the logarithmic loss.  In all other cases with $k \leq n-2$, we may interpolate \eqref{L1Linf} and \eqref{L2L2} to obtain
\begin{equation}\label{interpk}
\|W_tW_s^*h\|_{L^q_{y'}L^2_{y_1}} \lesssim \mu^{2(\frac{n-1}{2}-\frac{k-1}{q})} (1+\mu|t-s|)^{-(\frac{n-1}{2}-\frac{k-1}{q})}\|h\|_{L^{q'}_{y'}L^2_{y_1}}
\end{equation}
and use that $(1+|s|)^{-(\frac{n-1}{2}-\frac{k-1}{q})} \in L^{q/2}(\RR)$ to get \eqref{dualitybound}.  The same argument works when $k=n-1$ and $\frac{2n}{n-1}<q < \infty$.  To handle the remaining cases when $k=n-1$, we use that
$$
\mu^{2(\frac{n-1}{2}-\frac{n-2}{q})} (1+\mu|t-s|)^{-(\frac{n-1}{2}-\frac{n-2}{q})}\lesssim \mu^{\frac{n-1}2 - \frac{n-2}q } |t-s|^{-\frac{n-1}2 + \frac{n-2}q}.
$$
Hence \eqref{dualitybound} follows from the Hardy-Littlewood-Sobolev inequality when $q= \frac{2n}{n-1}$.  When $2 \leq q < \frac{2n}{n-1}$, the right hand side is in $L^{q/2}_{loc}$ and Young's inequality gives \eqref{dualitybound}.

In what follows, we will denote the integral kernel of $W_tW_s^*$ as $K_{t,s}(y,z)$.
The bound \eqref{L1Linf} follows from the proofs of the bounds \cite[(3.5)]{SmC2} or \cite[(5.4), (7.2)]{smithsogge06} due to Smith and Smith-Sogge respectively.  Those works establish the uniform inequality
$$
\int |K_{t,s}(y,z)|\,dy_1 + \int |K_{t,s}(y,z)|\,dz_1 \lesssim \mu^{n-1}(1+\mu|t-s|)^{-\frac{n-1}{2}}.
$$
It thus suffices to prove \eqref{L2L2}.  Using that $(x,\xi)\mapsto\Theta_{r_1,r_2}(x,\xi)$ defines diffeomorphism which preserves $dx \wedge d\xi$, the kernel of $W_tW^*_s$ can be realized as (cf. \cite[p.127]{smithsogge06})
\begin{equation}\label{Kdef}
K_{t,s}(y,z)=\mu^{\frac n2}\int e^{i\langle \xi, \bar{z}-x\rangle-i\langle \xi_{s,t}, \bar{y}-x_{s,t}\rangle}\phi(\mu^{\frac 12}(\bar{z}-x))\phi(\mu^{\frac 12}(\bar{y}-x_{s,t})) \Gamma(\xi)\,dxd\xi
\end{equation}
with $(x_{s,t},\xi_{s,t})$ abbreviating $(x_{s,t}(x,\xi),\xi_{s,t}(x,\xi))$.  Here $\Gamma$ is a cutoff supported in a region of the form appearing in \eqref{wpsupport} which may be inserted since we are only interested in functions $\tilde{f}$ satisfying that condition.

Before proceeding further, we observe bounds on the bicharacteristic flow of $p$.

\begin{theorem}\label{thm:bichar}
Suppose $(x,\xi) \in \RR^{2n}$ with $\xi$ in the set defined by \eqref{epscone}.  Let $\Theta_{t,s}(x,\xi)$ be as in \eqref{duhamel}, that is, $\Theta_{t,s}(x,\xi)|_{t=s} = (x,\xi)$ and
\begin{equation}\label{hamilton}
\prtl_s x_{t,s}(x,\xi) = d_\xi p(s, \Theta_{t,s}(x,\xi) ), \qquad \prtl_s \xi_{t,s}(x,\xi) = -d_x p(s, \Theta_{t,s}(x,\xi) ).
\end{equation}
Then for $t,s \in [0,1]$, first partials of $x_{t,s}(x,\xi)$, $\xi_{t,s}(x,\xi)$ in $x,\xi$ satisfy
\begin{equation}\label{dxbounds}
\left|d_x x_{t,s}-I\right| + \left|d_x \xi_{t,s} \right| \lesssim c_0|t-s|,
\end{equation}
\begin{equation}\label{dxibounds}
\left|d_\xi x_{t,s}(x,\xi) - \int_t^s d_\xi d_\xi p(\Theta_{r,t}(x,\xi))\,dr\right| + \left|d_\xi \xi_{t,s}(x,\xi)-I \right| \lesssim c_0|t-s|^2.
\end{equation}
\end{theorem}

\begin{proof}
Differentiating the equations \eqref{hamilton} gives
\begin{equation*}
\prtl_r
\begin{bmatrix}
dx_{t,r}\\
d\xi_{t,r}
\end{bmatrix} =
M(r,x_{t,r},\xi_{t,r})\begin{bmatrix}
dx_{t,r}\\
d\xi_{t,r}
\end{bmatrix} ,
\qquad
\text{where } M =
\begin{bmatrix}
d_x d_\xi p & d_\xi d_\xi p\\
-d_\xi d_x p & d_x d_x p\\
\end{bmatrix}.
\end{equation*}
By Gronwall's inequality and the bounds \eqref{papprox} we have
$$
|d_x x_{t,r}-I| + |d_x\xi_{t,r}| \lesssim 1, \qquad |d_\xi x_{t,r}| + |d_\xi \xi_{t,r}-I| \lesssim 1,
$$
and substituting these bounds back into the integral equation for $dx_{t,r}, d\xi_{t,r}$ implies the theorem.
\end{proof}

This type of argument can also be used to bound higher order derivatives of $x_{t,s}, \xi_{t,s}$, see e.g. \eqref{xiderivs} below.  Such bounds are used in the proof of the next theorem.  It is due to Smith-Sogge (see \cite[Theorem 5.4]{smithsogge06} which obtains bounds on $K_{t,s}$ under the assumption that $\Gamma$ is a smooth cutoff to a (possibly) smaller set.
\begin{theorem}\label{thm:theorem54}
Suppose $\bar{\theta}= \min(1, \mu^{-\frac 12}|t-s|^{-\frac 12})$ and the smooth cutoff $\Gamma$ in \eqref{Kdef} is supported in a set contained in \eqref{wpsupport} of the form
\begin{equation}\label{Gammasupp}
\supp(\Gamma) \subset \{ \xi: |\xi/|\xi|-\eta| \lesssim \bar{\theta}\}
\end{equation}
for some unit vector $\eta\in \mathbb{S}^{n-1}$.  Let $(x_{t,s},\nu_{t,s})=\Theta_{t,s}(\bar{z},\eta)$.  Then $K_{t,s}$ satisfies the pointwise bounds
\begin{equation}\label{kthetabarbound}
|K_{t,s}(y,z)| \lesssim \mu^n\bar{\theta}^{n-1}(1+\mu\bar{\theta} |\bar{y}-x_{t,s}| + \mu|\langle \nu_{t,s}, \bar{y}-x_{t,s}\rangle|)^{-N}.
\end{equation}
\end{theorem}

Observing that $\mu^{n-k}(1+\mu|t-s|)^{-\frac{n-k}{2}} \approx \min(\mu^{n-k}, \mu^{\frac{n-k}{2}}|t-s|^{-\frac{n-k}{2}} )$, we begin treating the case $|t-s| \leq \mu^{-1}$, that is, the case where the first quantity is smaller.  In this case, we apply \eqref{kthetabarbound} in Theorem \ref{thm:theorem54} with $\bar{\theta} = 1$ and $\eta = -e_1$ to obtain
\begin{equation*}
|K_{t,s}(y,z)| \lesssim \mu^n(1+\mu|\bar{y}-x_{t,s}(z,-e_1)| )^{-N}
\end{equation*}
which gives the first half of \eqref{symmetricyoungs} below.  Making the measure preserving change of variables $(x,\xi) \mapsto (x_{t,s}(x,\xi),\xi_{t,s}(x,\xi))$ in \eqref{Kdef}, we may reverse the roles of $y$ and $z$ in Theorem \ref{thm:theorem54} to obtain an analogous bound which yields
\begin{equation}\label{symmetricyoungs}
\int |K_{t,s}(y,z)|\, dy + \int |K_{t,s}(y,z)|\, dz \lesssim \mu^{n-k}
\end{equation}
(strictly speaking, the change of variables replaces $\Gamma(\xi)$ by $\Gamma(\xi_{t,s}(x,\xi))$, but this does not change the validity of the bounds in Theorem \ref{thm:theorem54}).

It now suffices to treat the more involved case where $\mu^{-1} < |t-s| \leq 1$, and for the remainder of this section we assume $t,s\in [0,1]$ are two fixed values satisfying this condition.  Using the notation suggested by Theorem \ref{thm:theorem54}, we set $\bar{\theta} = \mu^{-\frac 12} |t-s|^{-\frac 12}$ so that $\mu\bar{\theta}^2 |t-s|=1$. Using a partition of unity, we take a decomposition $K_{t,s} = \sum_j K^j$ where $K^j$ is defined by replacing $\Gamma$ in \eqref{Kdef} by a smooth cutoff $\Gamma_j$, with $\Gamma_j$ supported in a set of the form $|\xi/|\xi|-\eta^j| \lesssim \bar{\theta}$ and $\{\eta^j\}$ is a collection of unit vectors in the cone $\{-\xi_1 \gtrsim \veps^{-1}|(\xi',\xi'')| \}$ separated by a distance of at least $\approx\bar{\theta}^{-1}$.
In particular, we may assume that for fixed $j$
\begin{equation}\label{etasep}
\sum_{l}(1+\bar{\theta}^{-1}|\eta^j - \eta^l|)^{-(n+1)} \lesssim 1.
\end{equation}

Let $T_j$ be the operator defined by $(T_jh)(y) = \int K^j(y,z) h(z)\,dz$ and observe that since $|(\nu^j)_1| \approx 1$, \eqref{kthetabarbound} in Theorem \ref{thm:theorem54} with $\eta = \eta^j$ gives
$$
\int |K^j(y,z)| \,dy \lesssim \mu^{n-k}\bar{\theta}^{n-k}.
$$
By the same symmetry argument used in \eqref{symmetricyoungs}, we now have
$$
\|T_j h\|_{L^2_y} \lesssim \mu^{n-k}\bar{\theta}^{n-k}\|h\|_{L^2_y} = \mu^{\frac{n-k}{2}}|t-s|^{-\frac{n-k}{2}}\|h\|_{L^2_y}
$$
(though in what follows, it is convenient to express the bounds in terms of $\mu$, $\bar{\theta}$).

We claim that there exists a constant $C$ such that if $\bar{\theta}^{-1}|\eta^j-\eta^l| \geq C$, then
$$
\|T_{l}^*T_{j}\|_{L^2 \to L^2} + \|T_{l}T_{j}^*\|_{L^2 \to L^2} \lesssim \mu^{2(n-k)}\bar{\theta}^{2(n-k)}(1+\bar{\theta}^{-1}|\eta^j-\eta^l|)^{-N}.
$$
Since $W_t W_s^* = \sum_j T_j$, Cotlar's lemma then implies \eqref{L2L2}. Furthermore, we focus on the bound for $T_{l}^*T_{j}$ as symmetric argument yields the bound on $T_{l}T_{j}^*$.  Set
$$
J_{j,l}(z,w) = \int \overline{K^l(y,z)}K^j(y,w)\,dy.
$$
We will show that for $\bar{\theta}^{-1}|\eta^j-\eta^l| \geq C$,
\begin{equation}\label{Jclaim}
|J_{j,l}(z,w)|\lesssim \mu^{2n-k}\bar{\theta}^{2n-1-k} (1+\mu\bar{\theta}|z-w|+ \mu|\langle\eta^l,\bar{z}-\bar{w}\rangle|+
\bar{\theta}^{-1}|\eta^j-\eta^l|)^{-N}.
\end{equation}

The proof of \eqref{Jclaim} varies based on whether $|(\eta^j_1-\eta^l_1,\dots,\eta^j_k-\eta^l_k)| \geq |(\eta^j-\eta^l)''|$ or the opposite inequality holds.  In the first case, we write
\begin{multline}\label{firstlockernel}
J_{j,l}(z,w)=
\mu^{\frac n2}\int \int \left(\int e^{i\langle \xi, \bar{y}-x\rangle-i\langle \tilde{\xi}, \bar{y}-\tilde{x}\rangle}\phi(\mu^{\frac 12}(\bar{y}-x))\phi(\mu^{\frac 12}(\bar{y}-\tilde{x}))\,dy\right)\\
\times\psi(z,w,x,\xi,\tilde{x},\tilde{\xi}) \Gamma_j(\xi)\Gamma_l(\tilde{\xi})\,dx d\xi\,d\tilde{x}d\tilde{\xi}
\end{multline}
where $(\tilde{x},\tilde{\xi})$ denote the variables in the integral defining $K_l$ and $\psi$ is a function independent of $y$. The $y$ integral in parentheses is a constant multiple of
\begin{equation}\label{firstloc}
\int e^{i\tilde{\psi}}\;
\widehat{\phi}(\mu^{-\frac 12}((\zeta_1,\zeta',\zeta'')-\xi))\widehat{\phi}(\mu^{-\frac 12}((\zeta_1,\zeta',\tilde{\zeta}'')-\tilde{\xi}))\,d\zeta_1 d\zeta' d\zeta'' d\tilde{\zeta}''
\end{equation}
where $\tilde{\psi}$ is some real valued phase function.
Since $\supp(\widehat{\phi})$ is contained in the unit ball and $2|(\eta^j_1-\eta^l_1,\dots,\eta^j_k-\eta^l_k)| \geq |\eta^j-\eta^l|$, this integral vanishes if $\bar{\theta}^{-1}|\eta^l-\eta^j|\geq C$  as this implies that $|(\xi_1-\tilde{\xi}_1,\dots,\xi_k-\tilde{\xi}_k)| \gtrsim C\mu\bar{\theta} \geq C\mu^{\frac 12}$.

We now turn to the case where $|(\eta^j)''-(\eta^l)''|\geq |(\eta^j_1-\eta^l_1,\dots,\eta^j_k-\eta^l_k)|$.  In this case, we use \eqref{kthetabarbound} in Theorem \ref{thm:theorem54} to bound $|K_l|$, $|K_j|$ individually.  After some minor manipulations, this yields
\begin{multline}\label{Jbound}
|J_{j,l}(z,w)|\lesssim \mu^{2n}\bar{\theta}^{2(n-1)} \times\\
\int (1+\mu\bar{\theta}|\bar{y}-x_{t,s}(\bar{w},\eta^j)|+
\mu\bar{\theta}|\bar{y}-x_{t,s}(\bar{z},\eta^l)|+
\mu|\langle\nu_{t,s}(\bar{z},\eta^l),\bar{y}-x_{t,s}(\bar{z},\eta^l)\rangle|)^{-6N}\times\\
(1+\mu|\langle\nu_{t,s}(\bar{w},\eta^j),\bar{y}-x_{t,s}(\bar{w},\eta^j)\rangle-
\langle\nu_{t,s}(\bar{z},\eta^l),\bar{y}-x_{t,s}(\bar{z},\eta^l)\rangle|)^{-N}\,dy
\end{multline}

We take $3N$ of the powers in the first factor of the integrand on the right and claim that up to implicit constants, it is bounded above by
\begin{equation}\label{nondirectional}
(1+\mu\bar{\theta}|z-w|+
\bar{\theta}^{-1}|\eta^j-\eta^l|)^{-3N}
\end{equation}
To see this, first observe that the $3N$ powers from the integrand are dominated by
$$
(1+\mu\bar{\theta}|x_{t,s}(\bar{z},\eta^l)-x_{t,s}(\bar{w},\eta^j)|+
64\mu\bar{\theta}|x_{t,s}''(\bar{z},\eta^l)-x_{t,s}''(\bar{w},\eta^j)|)^{-3N}.
$$
By the bounds \eqref{dxbounds}, \eqref{dxibounds} in Theorem \ref{thm:bichar}, we have
\begin{equation}\label{noprimes}
|x_{t,s}(\bar{z},\eta^l)-x_{t,s}(\bar{w},\eta^j)| \geq \frac 34|z-w| - 2|t-s||\eta^l-\eta^j|
\end{equation}
provided $c_0$ and $\veps$ are taken sufficiently small.  Next we use that
$$
\left|x_{t,s}''(\bar{w},\eta^j)- x_{t,s}''(\bar{z},\eta^l)\right| \geq \left|x_{t,s}''(\bar{w},\eta^j)- x_{t,s}''(\bar{w},\eta^l)\right| -\left|x_{t,s}''(\bar{w},\eta^l)- x_{t,s}''(\bar{z},\eta^l)\right|.
$$
To bound the second term on the right, we use that as a consequence of \eqref{dxbounds} the $(n-k) \times n$ matrix $d_x x_{t,s}''$ satisfies
$$
\left|d_x x_{t,s}''-\left[ 0 \;\;I_{n-k}\right]\right| \lesssim c_0|t-s|.
$$
Recalling that $\bar{w} = (w,0)$, $\bar{z} = (z,0)$, this gives
$$
\left|x_{t,s}''(\bar{w},\eta^l)- x_{t,s}''(\bar{z},\eta^l)\right| \lesssim c_0|t-s||z-w|.
$$
We now use \eqref{papprox}, \eqref{dxibounds} to get that $d_\xi x_{t,s}''(x,\xi)$ is the $(n-k) \times n$ block matrix
$$
(s-t)(\xi_1^2-|(\xi',\xi'')|^2)^{-3/2}
\begin{bmatrix}
\xi_1 \xi'' & -\xi''(\xi')^T & -\Big((\xi_1^2-|(\xi',\xi'')|^2)I_{n-k}+\xi'' (\xi'')^T\Big)
\end{bmatrix}
$$
plus an error term which is $\mathcal{O}(c_0|t-s|)$.  Here $\xi''$ is taken to be a column vector.  Using that $|(\eta^l-\eta^j)''| \geq |\eta^l-\eta^j|/2$ and $|(\xi',\xi'')|\lesssim \veps|\xi_1 |$, we have that
$$
\left|x_{t,s}''(\bar{w},\eta^j)- x_{t,s}''(\bar{w},\eta^l)\right| \geq \frac{|t-s|}{8}|\eta^l-\eta^j|.
$$
In summary, we have that for some uniform constant $M$,
\begin{equation}\label{doubleprimes}
64\left|x_{t,s}''(\bar{w},\eta^j)- x_{t,s}''(\bar{z},\eta^l)\right| \geq 8|t-s||\eta^l-\eta^j| - Mc_0|t-s||z-w|
\end{equation}
By taking $c_0$ sufficiently small, the negative term in \eqref{doubleprimes} can be absorbed by the first term in \eqref{noprimes} and vice versa, which shows \eqref{nondirectional}.

We now turn to the second factor in the integrand of \eqref{Jbound}.  The triangle inequality gives
\begin{equation*}
\mu\left|\langle\nu_{t,s}(\bar{w},\eta^j),\bar{y}-x_{t,s}(\bar{w},\eta^j)\rangle-
\langle\nu_{t,s}(\bar{z},\eta^l),\bar{y}-x_{t,s}(\bar{z},\eta^l)\rangle\right| \geq \mu |\langle\eta^j,\bar{z}-\bar{w}\rangle| -  E
\end{equation*}
with
\begin{multline*}
E=
\mu\left|\nu_{t,s}(\bar{z},\eta^l)-\nu_{t,s}(\bar{w},\eta^j)\right|
\left|\bar{y}-x_{t,s}(\bar{z},\eta^l)\right|\\
+\mu\left|\langle\nu_{t,s}(\bar{w},\eta^j),x_{t,s}(\bar{z},\eta^l)-x_{t,s}(\bar{w},\eta^j)\rangle
-\langle\eta^j,\bar{z}-\bar{w}\rangle\right|.
\end{multline*}
We claim that
\begin{equation}\label{Rbound}
E \lesssim (\mu\bar{\theta})^2\left|\bar{y}-x_{t,s}(\bar{z},\eta^l)\right|^2  + \bar{\theta}^{-2}|\eta^j-\eta^l|^2 + (\mu\bar{\theta})^2|z-w|^2 + 1.
\end{equation}
The error induced by $E$ can thus be absorbed by $2N$ of the powers in \eqref{nondirectional} and $2N$ of the powers in the first factor in \eqref{Jbound}.  This concludes the proof of \eqref{Jclaim} as the remaining $N$ powers of the first factor in \eqref{Jbound} can be used to integrate in $y$.

To bound the first term in $E$, we use the geometric-arithmetic mean inequality and observe that the bounds on $d_x \xi_{t,s}$, $d_\xi \xi_{t,s}$ in Theorem \ref{thm:bichar} give
$$
\bar{\theta}^{-1}\left|\nu_{t,s}(\bar{z},\eta^l)-\nu_{t,s}(\bar{w},\eta^j)\right| \lesssim \bar{\theta}^{-1}|z-w| + \bar{\theta}^{-1}|\eta^j-\eta^l|.
$$
Since $\bar{\theta}^{-1} \leq \mu \bar{\theta}$ when $|t-s| \leq 1$, this is seen to be bounded by the right hand side of \eqref{Rbound}.  Using that $\mu\bar{\theta}^2|s-t|=1$ and $\mu \leq (\mu\bar{\theta})^2$ the rest of \eqref{Rbound} follows from
$$
\left|\langle\nu_{t,s}(\bar{w},\eta^j),x_{t,s}(\bar{z},\eta^l)-x_{t,s}(\bar{w},\eta^j)\rangle
-\langle\eta^j,\bar{z}-\bar{w}\rangle\right| \lesssim |z-w|^2 + \bar{\theta}^2|s-t|
$$
which can be seen by differentiating the expression on the left in $s$ see \cite[p.133]{smithsogge06}.

\section{Curved Submanifolds}\label{sec:curved}
In this section, we prove the bound  \eqref{wave2ndorder} with $q=2$, $\delta = \frac 16 (1+\beta)$ which implies Theorem \ref{thm:curved}.  In contrast to the previous section, it will be more convenient to work with an equation which is hyperbolic in $t$ rather than in $x_1$.  To this end, we simply set $q^{\pm}(x,\xi) = \pm \left(\sum_{i,j}\g^{ij}(x)\xi_i\xi_j\right)^{\frac 12}$ and $p^\pm(\cdot,\xi) = S_{c^2\mu^{1/2}}q^\pm(\cdot,\xi)$.  As a consequence, we vary the notational conventions slightly so that if $x \in \RR^n$, we denote $x'=(x_1,\dots,x_{n-1})\in \RR^{n-1}$ so that $x=(x',x_n)$.  All other conventions will carry over as before.

Following reductions similar to the previous section, it suffices to show that
\begin{equation*}
\|u_\mu^\pm\|_{L^2((0,1) \times \RR^{n-1} \times \{0\})} \lesssim \mu^{\frac 16 (1+\beta)}\left(\|u_\mu^\pm\|_{L^2(\RR^{n+1})}+
\mu^{-1}\|G_\mu^\pm\|_{L^2(\RR^{n+1})} \right)
\end{equation*}
where $G_\mu^\pm=(-i\prtl_t + p^\pm(x,D))u_\mu^\pm$.  As before, it suffices to treat the $u_\mu^-$ so we suppress the superscripts below.

The wave packet transform from above can also be used here, and after following the initial reductions in $\S3$, it suffices to show that the propagator
\begin{align*}
W\tilde{f}(t,y) &= T^*_\mu (\tilde{f}\circ \Theta_{0,t})(\bar{y})\\
&=\mu^{\frac n4}\int e^{i\langle \xi_{t,0}(x,\xi), \bar{y}-x_{t,0}(x,\xi)\rangle}\phi(\mu^{\frac 12}(\bar{y}-x_{t,0}(x,\xi)))\tilde{f}(x,\xi)\,dxd\xi
\end{align*}
satisfies
\begin{equation}\label{curvedWbound}
\|W\tilde{f}\|_{L^2_{t,y}} \lesssim \mu^{\frac 16(1+\beta)}\|\tilde{f}\|_{L^2_{x,\xi}}, \qquad \beta = \frac{\sigma}{1-\sigma}<\frac 12.
\end{equation}
where $\tilde{f}$ is supported in a region of the form $\{\xi: |\xi|\approx \mu, |\xi_1/|\xi|-(-e_1)| \lesssim \veps \}.$
In this section, the map $\Theta_{t,s}$ is determined by the new value of $p$ and hence $\Theta_{t,s}=\Theta_{t-s,0}$.  Given \eqref{rapiddku}, we may assume $(t,y)$ are restricted to $(0,1) \times (-1,1)^{n-1}$, that is, we bound $L^2((0,1) \times (-1,1)^{n-1})$ norm of $W\tilde{f}$.  We now exploit the property \eqref{negdefgmu}.

\begin{lemma}
Let $(x(t),v(t))$ be a solution to the geodesic equation in tangent space
\begin{equation}\label{geodeqn}
\frac{dx_k}{dt} = v^k(t) \qquad \frac{dv^k}{dt} = -v^i(t)v^j(t)\Gamma^k_{ij}(x(t))
\end{equation}
relative to the Christoffel symbols defined by $\g_{\mu^{1/2}}$ (with summation convention in effect). Suppose further that $(x(t),v(t))$ is defined for $t \in [-1,1]$ and that the geodesic has unit speed in that $|v(t)|_{\g_{\mu^{1/2}}} \equiv 1$.  If $v(t)$ further satisfies $|v^n(t)| \lesssim \veps$, where $\veps$ is sufficiently small, then
there exists a uniform constant $c_1$ such that the $n$-th component of the velocity satisfies
\begin{equation}\label{parabxi}
c_1 \mu^{-\beta}|t| \leq v^n(t)-v^n(0) \lesssim c_0 \mu^{-\beta}|t|.
\end{equation}
Furthermore, the difference between $x_n(t)$ and its linearization about 0 satisfies
\begin{equation}\label{linearxn}
\left| x_n(t)-x_n(0)-v^n(0)t \right| \lesssim c_0\mu^{-\beta}|t|^2.
\end{equation}
\end{lemma}
\begin{proof}
If $\veps$ is sufficiently small relative to the $c_1$ appearing in \eqref{negdefgmu}, we have that $-v^i(t)v^j(t)\Gamma^n_{ij}(x(t))$ is uniformly bounded from above and below.
Adjusting the constant $c_1$, the bound \eqref{parabxi} is thus a consequence of the integral equations arising from \eqref{geodeqn}.  The integral equation for $x_n(t)$ similarly gives \eqref{linearxn}.
\end{proof}

Recall that solutions to \eqref{geodeqn} are naturally associated to curves $(x(t),\xi(t))$ in the cotangent bundle by the identification $v^k(t) =\g^{kl}_{\mu^{1/2}}(x(t))\xi_l(t)$.  The curves in phase space are solutions to the Hamiltonian equations
$$
\frac{dx}{dt} = d_\xi H,  \qquad \frac{d\xi}{dt} = -d_x H, \qquad H(x,\xi) = \frac 12 \g^{ij}_{\mu^{1/2}}\xi_i\xi_j.
$$
With this in mind, we define $a(x,\xi) = g^{nm}(x)\xi_m = \prtl_{\xi_n}H$ where again the summation convention is used.  If $(x_{t,s}(x,\xi),\xi_{t,s}(x,\xi))$ were integral curves of the Hamiltonian vector field determined by $q= \sqrt{\g^{ij}\xi_i\xi_j}$, we would have that $a(x_{t,s},\xi_{t,s}) = |\xi|_{\g_{\mu^{1/2}}}v_n(s-t)$ where $v_n(r)$ is the $n$-th component of the velocity vector in \eqref{geodeqn} at time $r$ with initial data satisfying $x_k(0)=x_k$, $v^k(0) = \left(\g^{kl}_{\mu^{1/2}}(x)\xi_l\right)/|\xi|_{\g_{\mu^{1/2}}}$, $|v(0)|_{\g_{\mu^{1/2}}}=1$.  However, in the solution operator $W$ under consideration, the $(x_{t,s},\xi_{t,s})$ are integral curves of the Hamiltonian vector field determined by $p(\cdot,\xi) = S_{c^2\mu^{1/2}}q(\cdot,\xi)$.  Given the following bounds for $|\xi|\approx \mu$
$$
\left|\prtl^\gamma_\xi(p-q)(x,\xi)\right| \lesssim \mu^{-1}, \qquad \left|\prtl^\gamma_x(p -q)(x,\xi)\right| \lesssim c_0 \mu^{\frac 12},
$$
we can use Gronwall's inequality to approximate the integral curves of $d_\xi p \cdot d_x -d_x p \cdot d_\xi$ by those of $d_\xi q \cdot d_x -d_x q \cdot d_\xi$ and deduce that for $|\xi|\approx \mu$
\begin{equation}\label{aonpcurves}
a(x_{t,s}(x,\xi),\xi_{t,s}(x,\xi)) = |\xi|_{\g_{\mu^{1/2}}}v_n(t-s) + \mathcal{O}(\mu^{\frac 12}|t-s|)
\end{equation}
where $v_n(t-s)$ is as before.  By the same tack, \eqref{linearxn} gives that for $(x_{t,s})_n=\langle x_{t,s},e_n\rangle$,
\begin{equation}\label{linearxnp}
\left| (x_{t,s})_n(x,\xi)-x_n-|\xi|_{\g_{\mu^{1/2}}}^{-1} a(x,\xi)(t-s) \right| \lesssim c_0\mu^{-\beta}|t-s|^2 + \mu^{-\frac 12}|t-s|.
\end{equation}

Let $N_\mu$, $n_\mu$ be integers such that $N_\mu \approx \log_2(\mu^{\frac 13(1+\beta)})$, $n_\mu \approx \log_2(\mu^{\beta})$ and take a smooth partition of unity $\{\Gamma_j(r)\}_{j=n_\mu}^{N_\mu}$ on $\RR$ satisfying
\begin{align*}
\supp(\Gamma_{n_\mu}) &\subset \{r \in \RR: |r| \geq \mu 2^{-n_{\mu}-2}\},\\
\supp(\Gamma_j) &\subset \{r \in \RR: |r|  \in [\mu 2^{-j-2}, \mu2^{-j+2}]\}, &  n_\mu < j < N_\mu\\
\supp(\Gamma_{N_\mu}) &\subset \{r \in \RR: |r| \leq \mu 2^{-N_{\mu}+2} \}
\end{align*}
For each $n_\mu \leq j \leq N_\mu$, we define
$$
W^{j}\tilde{f}(t,y) = \mu^{\frac n4}\int e^{i\langle \xi_{t,0}, \bar{y}-x_{t,0}\rangle}\phi(\mu^{\frac 12}(\bar{y}-x_{t,0}))\Gamma_{j}(a(x_{t,0},\xi_{t,0}))\tilde{f}(x,\xi)\,dxd\xi
$$
and as before, we let $W^{j}_t\tilde{f}(y) =W^{j}\tilde{f}(r,y)|_{r=t}$.
It suffices to show that
\begin{equation}\label{wjopbound}
\|W^j\tilde{f}\|_{L^2_{t,y}} \lesssim 2^{\frac j2}\|\tilde{f}\|_{L^2_{x,\xi}}.
\end{equation}

When $\beta=0$ the decomposition above is consistent with earlier treatments of FIOs whose canonical relations possess two-sided fold (see e.g. \cite{cuccagna}).  Indeed, for an FIO determined by the classical Lax parametrix, the singularities of the right projection of the canonical relation are determined by $a(x_{t,0},\xi_{t,0})=0$ and it is effective to take dyadic decomposition in $a(x_{t,0},\xi_{t,0})/\mu$ in scales $1 \geq 2^{-j} \geq \mu^{-\frac 13}$.  For $\beta >0$, scaling considerations relating to the dilation of variables $x \mapsto \lambda^{-\sigma }x$ in \S2, then suggests that the dyadic scales should not be finer $\mu^{-\frac 13(1+\beta)}$.  In our circumstance, we can view
the splitting of $|a(x_{t,0},\xi_{t,0})|/\mu$ into scales less than and greater than $\mu^{-\frac 13(1+\beta)}$ as a decomposition into tangential and nontangential momenta respectively.  It can be seen that this threshold gives the largest scale at which our estimate for tangential momenta \eqref{wNmudualbound} is effective.  At the same time, restricting nontangential momenta to scales at least this size allows us to achieve an appreciable gain in the bounds for $W^{j}$ by using the linear approximation of phase space transport in \eqref{linconseq} below.  The selection of $n_\mu$ is more technical, its choice is based on the fact that for $|a(x,\xi)|/\mu\geq \mu^{-\beta}$, the $(\xi_{t,0})_n$ component of the Hamiltonian flow can be linearized over a unit time scale.

Let $\omega_n$ be the unit vector pointing in the direction of $(\g^{n1}(\bar{z}),\dots,\g^{nn}(\bar{z}))$ and $B$ denote the projection matrix onto the subspace orthogonal to $\omega_n$.  Given the decomposition above, we will need to consider at the following class of integrals more general than those in Theorem \ref{thm:theorem54}
\begin{multline}\label{wjkernel}
K_{t,s}(y,z)=\mu^{\frac n2} \iint e^{i\langle \xi, \bar{z}-x\rangle-i\langle \xi_{t,s}, \bar{y}-x_{t,s}\rangle}\phi(\mu^{\frac 12}(\bar{z}-x))\phi(\mu^{\frac 12}(\bar{y}-x_{t,s}))\\
\times\widetilde{\Gamma}(\xi) \Gamma_{j}(a(x,\xi))\Gamma_{j}(a(x_{t,s},\xi_{t,s})) \,dxd\xi
\end{multline}
where $\Gamma_j$ is defined as above with $n_\mu \leq j \leq N_\mu$ and
\begin{equation}\label{tildegamma}
\supp(\widetilde{\Gamma}) \subset \{ \xi: |\xi| \approx \mu, |\xi_1/|\xi| -(-e_1)| \lesssim \veps, \left| B\xi/|B\xi| - \eta\right| \lesssim \bar{\theta}\},
\end{equation}
for some unit vector $\eta$ orthogonal to $\omega_n$.  In particular, if $\bar{\theta}=1$, $W^j_t(W^j_s)^*$ takes this form.  Our first task to observe a generalization of Theorem \ref{thm:theorem54}.

\begin{theorem}\label{thm:wjbounds}
Suppose $\bar{\theta} = \min(1, \mu^{-\frac 12}|t-s|^{-\frac 12}) \geq 2^{-j}$ and $K_{t,s}(y,z)$ is defined by \eqref{wjkernel}, \eqref{tildegamma}.  Let $\zeta$ denote a fixed vector in the support of in the $\xi$-support of $\widetilde{\Gamma}(\cdot) \Gamma_{j}(a(\bar{z},\cdot))$ and $w_{t,s} = x_{t,s}(\bar{z},\zeta)$, $\nu_{t,s}=\xi_{t,s}(\bar{z},\zeta)/|\xi_{t,s}(\bar{z},\zeta)|$.
Then $K_{t,s}(y,z)$ satisfies the bounds
$$
|K_{t,s}(y,z)| \lesssim \mu^{n}\bar{\theta}^{n-2}2^{-j} (1+\mu\bar{\theta}|B\cdot(\bar{y}-w_{t,s})| + \mu|\langle \nu_{t,s}, \bar{y}-x_{t,s}\rangle|)^{-N}.
$$
\end{theorem}

\begin{proof}
The proof is only a slight modification of the argument in \cite[p.152]{smithsogge06} and hence we only outline the significant differences.  Indeed, the only alteration is that in our case, $a(x,\xi)$ replaces $\xi_n$ and the factor $\Gamma_{j}(a(x_{t,s},\xi_{t,s}))$ is also present.  Let $\omega_1,\dots, \omega_{n}$ be an orthonormal basis on $\RR^n$ containing $\omega_n$.  We then define the following vector fields which preserve the phase in \eqref{wjkernel}
\begin{align*}
L_0 &= \frac{1-i(\langle \xi, \bar{z}-x\rangle-\langle \xi_{t,s}, \bar{y}-x_{t,s}\rangle)\langle \xi,d_\xi\rangle}{1+|\langle \xi, \bar{z}-x\rangle-\langle \xi_{t,s}, \bar{y}-x_{t,s}\rangle|^2},\\
L_k &= \frac{1-i(\mu\bar{\theta})^2\langle \omega_k, \bar{z}-x-d_\xi \xi_{t,s}\cdot(\bar{y}-x_{t,s})\rangle \langle \omega_k,d_\xi\rangle}{1+\mu^2\bar{\theta}^2 |\langle \omega_k, \bar{z}-x-d_\xi \xi_{t,s}\cdot(\bar{y}-x_{t,s})\rangle|^2}, \qquad 1\leq k \leq n-1,
\end{align*}
and define $L_n$ analogously to $L_k$ above with $\omega_n$ replacing $\omega_k$ and $2^{-j}$ replacing $\bar{\theta}$.  The idea to is integrate by parts in \eqref{wjkernel} using these vector fields.  We display the following bounds on the derivatives of $\Theta_{t,s}(x,\xi)$ in $x,\xi$ due to Smith-Sogge
\cite[(5.6), (5.7), (5.11), (5.12)]{smithsogge06}
\begin{align*}
|d_x^2 x_{t,s}| & \lesssim\langle \mu^{\frac 12}|t-s|\rangle, & |d^2_x \xi_{t,s}|& \lesssim \mu^{\frac 12},\\
|d_x d_\xi x_{t,s}| & \lesssim |t-s|\langle\mu^{\frac 12}|t-s|\rangle, & |d_x d_\xi \xi_{t,s}|& \lesssim \langle\mu^{\frac 12}|t-s|\rangle,\notag
\end{align*}
\begin{equation}\label{xiderivs}
|d^k_\xi x_{t,s}|+ |d^k_\xi \xi_{t,s}| \lesssim |t-s|\langle \mu^{\frac 12}|t-s|\rangle^{k-1}, \quad k \geq 2,
\end{equation}
\begin{equation*}
|(\xi\cdot d_\xi)^j (\mu\bar{\theta}d_\xi)^\alpha\mu^{\frac 32} \bar{\theta}d_\xi x_{t,s}| \lesssim 1,
\quad |(\xi\cdot d_\xi)^j (\mu \bar{\theta}d_\xi)^\alpha\mu\bar{\theta}\langle d_\xi \xi_{t,s}, \bar{y}-x_{t,s}\rangle| \lesssim \langle \mu^{\frac 12}|\bar{y}-x_{t,s}|\rangle,
\end{equation*}
where the last one is valid for $j+|\alpha| \geq 1$.  In \cite{smithsogge06}, these bounds were used to prove Theorem \ref{thm:theorem54} above and the aforementioned estimates.

Here the first crucial matter is to observe that the result of applying powers of the differential operators $\langle \xi,d_\xi\rangle$ and $\mu\bar{\theta}\langle \omega_k,d_\xi\rangle$ for $k=1,\dots,n-1$ to $\Gamma_{j}^{(i)}(a(x,\xi))$, $\Gamma_{j}^{(i)}(a(x_{t,s},\xi_{t,s}))$ is dominated by the other factors in the integrand.  Powers of $\langle \xi, d_\xi\rangle$ are easily handled by homogeneity.  Differentiating $\Gamma_{j}^{(i)}$ yields a gain of $\mu^{-1}2^{j}$ while derivatives of $\bar{\theta}2^{j}a(x,\xi)$ in the direction of $\omega_k$ are
$$
\bar{\theta}2^{j} \langle \omega_k, d_\xi\rangle a(x,\xi) = \bar{\theta}2^{j}\langle \omega_k,\g^{nm}(x)-\g^{nm}(\bar{z})\rangle.
$$
Since $\bar{\theta}2^j \lesssim \mu^{\frac 13(1+\beta)}\ll \mu^{\frac 12}$, this is dominated by $\mu^{\frac 12}|x-\bar{z}|$.

For $\Gamma_{j}^{(i)}(a(x_{t,s},\xi_{t,s}))$, first consider a single power of $\bar{\theta}2^j\langle \omega_k,d_\xi\rangle$ on $a(x_{t,s},\xi_{t,s})$
\begin{multline}\label{aflow}
\bar{\theta}2^j \langle \omega_k,d_\xi\rangle a(x_{t,s}(x,\xi),\xi_{t,s}(x,\xi)) = \bar{\theta}2^j \langle \omega_k,d_\xi\rangle \left( \g^{nm}(x_{t,s})(\xi_{t,s})_m\right)=\\  \bar{\theta}2^j \left( d_x \g^{nm}(x_{t,s})\cdot \langle \omega_k, d_\xi x_{t,s}\rangle (\xi_{t,s})_m + \g^{nm}(x_{t,s})\langle \omega_k, d_\xi (\xi_{t,s})_m \rangle\right).
\end{multline}
The first term on the right is bounded as $|d_\xi x_{t,s}(x,\xi)||\xi_{t,s}|\lesssim |t-s|$ and $\bar{\theta}2^j |t-s| \ll 1$. For the second, we rewrite the sum in $m$ as
$$
\left(\g^{nm}(x_{t,s})- \g^{nm}(\bar{z})\right)\langle \omega_k, d_\xi (\xi_{t,s})_m \rangle + \g^{nm}(\bar{z})\langle \omega_k, d_\xi (\xi_{t,s})_m -e_m\rangle.
$$
The second term is $\mathcal{O}(|t-s|)$ and can be dominated as before.  For the first term we use that
$$
|x_{t,s}(x,\xi)-\bar{z}| \leq |x_{t,s}(x,\xi)-x|+|x-\bar{z}|.
$$
The first term here is $\mathcal{O}(|t-s|)$ and the second can be treated as above.  For higher derivatives of \eqref{aflow}, we simply use homogeneity and \eqref{xiderivs} to see that the result of applying $l$ additional powers of $\bar{\theta}2^j\langle \omega_k,d_\xi\rangle$ is bounded by $(\bar{\theta}2^j)^{l+1}|t-s|\mu^{-\frac l2}\ll 1$.

Integration by parts using $L_0,\dots,L_n$ gives that $|K_{t,s}(y,z)|$ is dominated by
\begin{multline}\label{firstKbound}
\mu^{\frac n2} \iint (1+\mu^{\frac 12}|\bar{z}-x|+\mu^{\frac 12}|\bar{y}-x_{t,s}|)^{-N} (1+\mu\bar{\theta}|B\cdot(\bar{z}-x- d_{\xi}\xi_{t,s}\cdot(\bar{y}-x_{t,s}))| \\ + \mu 2^j|\langle \omega_n,\bar{z}-x-d_{\xi}\xi_{t,s}\cdot(\bar{y}-x_{t,s})\rangle|+|\langle \xi, \bar{z} -x \rangle - \langle \xi_{t,s}, \bar{y}-x_{t,s}\rangle| )^{-N} \,d\xi dx
\end{multline}
and we may assume that the values of $\xi$ are restricted to $\xi \in \supp(\widetilde{\Gamma}(\cdot) \Gamma_{j}(a(x,\cdot)))$.

Now observe that if $\xi$ is such a vector and $\tilde{\xi}$, $\tilde{\zeta}$ are vectors in the direction of $\xi$, $\zeta$ normalized so that $|B\tilde{\xi}|=|B\tilde{\zeta}|=1$, then
\begin{equation}\label{conebounds}
|B(\tilde{\xi}-\tilde{\zeta})| \lesssim \bar{\theta}, \qquad |(I-B)(\tilde{\xi}-\tilde{\zeta})| \lesssim |\bar{z}-x| + 2^{-j}.
\end{equation}
The first of the two inequalities is evident from the support condition on $\widetilde{\Gamma}$, the second follows by observing that $|B\xi|,|B\zeta|\approx \mu$ and
$$
\g^{nm}(\bar{z})\tilde{\xi}_m -\g^{nm}(\bar{z})\tilde{\zeta}_m = \left(\g^{nm}(\bar{z})-\g^{nm}(x)\right)\tilde{\xi}_m + \g^{nm}(x)\tilde{\xi}_m -\g^{nm}(\bar{z})\tilde{\zeta}_m
$$
Given \eqref{conebounds}, the proof of \cite[(5.13)]{smithsogge06} goes through with only minor adjustments. Hence we have that \eqref{firstKbound} is further dominated by
\begin{multline}\label{secondKbound}
\mu^{\frac n2} \iint (1+\mu\bar{\theta}|B\cdot d_{\xi}\xi_{t,s}\cdot(\bar{y}-w_{t,s})|  + \mu 2^{-j}|\langle \omega_n,d_{\xi}\xi_{t,s}\cdot(\bar{y}-w_{t,s})\rangle|+|\langle \xi_{t,s}, \bar{y}-w_{t,s}\rangle| )^{-N}\\
\times(1+\mu^{\frac 12}|\bar{z}-x|+\mu^{\frac 12}|\bar{y}-x_{t,s}|)^{-N}  \,d\xi dx \end{multline}
where $\xi$ values are restricted as before.  Observe that since $\mu^{\frac 12} \ll \mu2^{-j} \leq \mu\bar{\theta}$ and $d_\xi \xi_{t,s}$ is invertible, the middle two terms in the first factor dominate $\mu^{\frac 12}|\bar{y}-w_{t,s}|$.

We next see that we may replace $\xi_{t,s}$  by $\xi_{t,s}(\bar{z},\zeta)$ in the expression $\langle \xi_{t,s}, \bar{y}-w_{t,s}\rangle$. Without loss of generality, we may assume that $|B\xi|=|B\zeta|$.  We note that
$$
|\langle\xi_{t,s}(x,\zeta) - \xi_{t,s}(\bar{z},\zeta), \bar{y}-w_{t,s} \rangle| \lesssim \mu |x-\bar{z}||\bar{y}-x_{t,s}|.
$$
It now remains to bound $|\langle\xi_{t,s}(x,\xi) - \xi_{t,s}(x,\zeta), \bar{y}-w_{t,s} \rangle|$.  We thus write
$$
\xi_{t,s}(x,\xi) - \xi_{t,s}(x,\zeta) = (\xi-\zeta)\cdot d_\xi \xi_{t,s}+ \mathcal{O}(|\xi-\zeta|^2\mu^{-\frac 12}|t-s|).
$$
For the first term here, note that
$$
(\xi-\zeta)\cdot d_\xi \xi_{t,s}= (\xi-\zeta)\cdot B\cdot d_\xi \xi_{t,s} + (\xi-\zeta)\cdot(I-B)\cdot d_\xi \xi_{t,s}
$$
Since $B$ is an orthogonal projection and
$$
|B\cdot(\xi-\zeta)| \lesssim \mu\bar{\theta}, \qquad |(I-B)\cdot(\xi-\zeta)| \lesssim \mu 2^{-j}+ \mu|x-\bar{z}|,
$$
the error induced by the first term here is dominated by the other terms in the integrand in \eqref{secondKbound}.  We then use $\mu \bar{\theta}^2|t-s| \leq 1$ bound the error term similarly.

Also, replacing $d_\xi \xi_{t,s}$ by the identity matrix in \eqref{secondKbound} yields an acceptable error as it is bounded by
$$
\mu\bar{\theta}|t-s||\bar{y}-x_{t,s}| \lesssim \mu^{\frac 12}|\bar{y}-x_{t,s}|.
$$

Finally, for each $x$, the region of integration in $\xi$ can be restricted to a set of volume $\approx \mu^n \bar{\theta}^{n-2}2^{-j}$, which is enough to conclude the proof.
\end{proof}

Note that by \eqref{c0}, we may assume that the difference between $B$ and projection onto the first $n-1$ coordinates yields an error which is no more than $\mathcal{O}(c_0)$.  Moreover, since $|\nu_{t,s}-e_1| \lesssim \veps + c_0$, we have that as a consequence of this theorem
\begin{equation}\label{integratedK}
\int |K_{t,s}(y,z)|\,dy \lesssim \mu 2^{-j}.
\end{equation}

We now begin the proof of \eqref{wjopbound} when $j=N_\mu$, claiming there exists $\tilde{c}_1$ such that
\begin{equation}\label{tgtseparation}
W^{N_\mu}_t (W^{N_\mu}_s)^*=0 \qquad \text{ whenever } \qquad |t-s| \geq \tilde{c}_1 \mu^\beta 2^{-N_\mu}.
\end{equation}
To see this, recall that the kernel of $W^{N_\mu}_t (W^{N_\mu}_s)^*$ is given by an integral of the form \eqref{wjkernel} with $\bar{\theta}=1$.  Since $\mu^{-\beta} \gg \mu^{-\frac 12}$, by \eqref{parabxi} and \eqref{aonpcurves}, there exists a constant $\tilde{c}_1$, inversely proportional to $c_1$ above, such that $|a(x_{t,s}(x,\xi),\xi_{t,s}(x,\xi))| \geq  \mu 2^{-N_\mu+2}$ whenever $\mu^{-\beta}|t-s| \geq \tilde{c}_12^{-N_\mu}$ and $\xi \in \supp(\Gamma_{N_\mu})$.

Turning to the case $|t-s| \leq \tilde{c}_1\mu^{\beta}2^{-N_\mu}$, take a collection of unit vectors $\eta^i$ orthogonal to $\omega_n$ and mutually separated by a distance $\approx \bar{\theta}$ so that \eqref{etasep} holds.  Now write $K_{t,s} = \sum_i K_i(y,z)$ where each $K_i$ is defined as in \eqref{wjkernel} with $\eta$ replaced by $\eta^i$.  Next observe that $|\eta^j-\eta^l| \lesssim |(\eta^j-\eta^l)'|$, which can be seen by noting that the linear map which projects the subspace orthogonal to $\omega_n$ onto its first $n-1$ components is invertible and depends continuously on $z$.  An adjustment of the almost orthogonality argument in \eqref{firstloc} thus shows that the operators $T_l^*T_j$, $T_lT_j^*$ vanish if $\bar{\theta}^{-1}|\eta^i-\eta^j| \geq C$ for some large $C$.  Observe that
\begin{equation}\label{youngsNmu}
\int |K_i(y,z)| \,dy + \int |K_i(y,z)| \,dz \lesssim \mu 2^{-N_\mu} .
\end{equation}
But the first half of this is a consequence of \eqref{integratedK} and the second half follows by symmetry and the same bound.  Indeed the theorem applies here as our assumption on $|t-s|$ means that $\bar{\theta} \gtrsim \mu^{-\frac 12(1+\beta)}2^{\frac{1}{2}N_\mu} \approx 2^{-\frac{3}{2}N_\mu+\frac{1}{2}N_\mu}=2^{-N_\mu}$.
The bound \eqref{wjopbound} now follows by duality since Young's inequality in $t,s$ gives
\begin{equation}\label{wNmudualbound}
\|W^{N_\mu}\left(W^{N_\mu}\right)^*\|_{L^2_{s,z} \to L^2_{t,y} } \lesssim \mu2^{-N_\mu}\cdot \mu^{\beta}2^{-N_\mu} \approx 2^{N_\mu}.
\end{equation}

For $n_\mu \leq j < N_\mu$, we take a partition of unity over $\RR^{n-1}$, $\sum_l \chi(y-l) \equiv 1$ such that the sum is taken over $l \in \mathbb{Z}^{n-1}$ and $\supp(\chi) \subset [-1,1]^{n-1}$. Use this to define
$$
\chi_l(y):=\chi(\mu^{-\beta}2^{j}y-l) \quad \text{and}\quad W^{j,l}\tilde{f}(t,y) := \chi_l(y)W^{j}\tilde{f}(t,y)
$$
and we consider only those $l$ such that $\supp(\chi_l)$ intersects $(-1,1)^n$.  By the support properties of $\chi$ we may take $C$ sufficiently large so that $(W^{j,m})^*W^{j,l}$ vanishes whenever $|l-m| \geq C$.  We next claim that we can take $C$ so that
\begin{equation}\label{cotlar}
\|W^{j,l}(W^{j,m})^*\|_{L^2 \to L^2} \lesssim \mu^{-N} \qquad \text{whenever } |l-m| \geq C.
\end{equation}
Since there is at most $\mathcal{O}(\mu^{\frac{n-1}{3}})$ of the $W^{j,l}$, the estimate \eqref{wjopbound} on $W^j$ will follow by Cotlar's lemma and Young's inequality provided we can show
\begin{equation}\label{ntanboundloc}
\|W^{j,l}_t (W^{j,l}_s)^*h\|_{L^2_{y}} \lesssim \mu 2^{-j}(1+\mu 2^{-2j}|t-s|)^{-2}\|h\|_{L^2_{y}}.
\end{equation}

In order to show \eqref{cotlar}, we can write the kernel of the operator, denoted by $K_{t,s}^{l,m}(y,z) $, as the product of $\chi_l(y)\chi_m(z)$ with an integral of the form \eqref{wjkernel} with $\bar{\theta}=1$.  Given the compact support of $K_{t,s}^{l,m}$ in $y$ and $z$ it suffices to show that this integral is dominated by $\mu^{-N}$ for any $N$.  Similar to the $j=N_\mu$ case, if $\xi \in \supp(\Gamma_j)$ and $|t-s| \geq \tilde{c}_1\mu^\beta 2^{-j}$ for some $\tilde{c}_1$ depending only on $c_1$, then $\Gamma_{j}(a(x_{t,s},\xi_{t,s}))=0$, meaning the kernel vanishes for such $t,s$.  When $|t-s| \leq \tilde{c}_1\mu^\beta 2^{-j}$, we use that
$$
|x_{t,s}(\bar{z},\xi)-x_{t,s}(x,\xi)|\lesssim |\bar{z}-x|
$$
to dominate the integral in \eqref{wjkernel} simply by
$$
\mu^{\frac n2}\iint (1+\mu^{\frac 12}|\bar{z}-x| + \mu^{\frac 12}|\bar{y}-x_{t,s}(\bar{z},\xi)|)^{-2N} \widetilde{\Gamma}(\xi)\Gamma_{j}(a(x,\xi))\,dxd\xi.
$$
Using the elementary estimate $|x_{t,s}(\bar{z},\xi) - \bar{z}| \leq 2|t-s|,$
we see that if $|l-m|\geq 2^4\tilde{c}_1$,
\begin{align*}
|\bar{y}-x_{t,s}(\bar{z},\xi)|&\geq | y - z| - 2|t-s|\\
&\geq \mu^\beta2^{-j-2}|l-m| - 2\tilde{c}_1\mu^\beta2^{-j}\geq \mu^\beta 2^{-j} \geq \mu^{-\frac 13}
\end{align*}
and hence $\mu^{\frac 12}|\bar{z}-x_{t,s}(\bar{y},\xi)|\gtrsim \mu^{\frac 16}$.  This implies the desired bound on $K_{t,s}^{l,m}(y,z) $.

We now turn to \eqref{ntanboundloc}.  It suffices to restrict attention to $|t-s| \leq \tilde{c}_1\mu^\beta 2^{-j}$, though this does not play a crucial role in the argument.  First consider the case where $t,s$ satisfy $|t-s| \leq \mu^{-1}2^{2j}$.  We begin by observing that a slight adjustment of the almost orthogonality argument in \eqref{firstloc} and preceding \eqref{wNmudualbound} allows us to assume that the kernel $K^{l,l}_{t,s}(y,z)$ of $W_t^{j,l}(W_s^{j,l})^*$ is the product of $\chi_l(y)\chi_l(z)$ and an integral of the form \eqref{wjkernel} with $\bar{\theta} = \min(1,\mu^{-\frac 12}|t-s|^{-\frac 12})$.  Indeed, reasoning as in \eqref{firstlockernel}, we are lead to consider the integral
$$
\int e^{i\langle \xi, \bar{y}-x\rangle-i\langle \tilde{\xi}, \bar{y}-\tilde{x}\rangle}\phi(\mu^{\frac 12}(\bar{y}-x))\phi(\mu^{\frac 12}(\bar{y}-\tilde{x}))\chi_l^2(y)\,dy.
$$
While this integral does not vanish when $\mu^{\frac 12} \ll \mu \bar{\theta} \leq |\xi-\tilde{\xi}|$, we may bound its absolute value by $C_N\mu^{-N}$ for any $N$, which is just as effective.  Indeed, we may take the Fourier transform similarly to \eqref{firstloc} and since the Fourier transform of $\chi^2_l$ is concentrated (though not localized) in a ball of radius $\mu^{-\beta}2^{j} \leq \mu^{\frac 13}\ll \mu^{\frac 12}$, the rapid decay in $\mu$ follows.  We now conclude \eqref{ntanboundloc} for $|t-s| \leq \mu^{-1}2^{2j}$ by applying \eqref{integratedK} and reasoning analogously to \eqref{youngsNmu}.

To show \eqref{ntanboundloc} when $|t-s| > \mu^{-1}2^{2j}$, we take the decomposition used in \S3, writing the kernel $K^{l,l}_{t,s} = \sum_i K_i$ with $K_i$ is defined by replacing the $\Gamma$ in \eqref{Kdef} by a smooth cutoff $\widetilde{\Gamma}_{j,i}$ to a region of the form
$$
\{\xi \in \supp(\tilde{\Gamma}(\cdot)\Gamma_j(a(x,\cdot)): -\xi_1 \approx \mu, |\xi/|\xi|-\eta^i| \lesssim \bar{\theta}\}, \qquad \bar{\theta} = \mu^{-\frac 12}|t-s|^{-\frac 12}
$$
where $\eta_i \in \mathbb{S}^{n-1}$.  As before, we assume that $\eta^i$ are separated so that \eqref{etasep} holds.  The estimates \eqref{kthetabarbound} in Theorem \ref{thm:theorem54} give
\begin{equation}\label{smsoki}
|K_i(y;z)| \lesssim \mu^{n}\bar{\theta}^{n-1}(1+\mu\bar{\theta}|\bar{y}-x_{t,s}^i| + \mu|\langle \nu_{t,s}^i, \bar{y}-x_{t,s}^i\rangle|)^{-N}
\end{equation}
with $x^i_{t,s}=x_{t,s}(\bar{z},\eta^i)$. We will show that
\begin{equation}\label{twojdecay}
|\bar{y}-x_{t,s}^i| \gtrsim 2^{-j}|t-s|.
\end{equation}
Together with our assumption on $t,s$ this gives $\mu^{\frac 12}2^{-j}|t-s|^{\frac 12} \lesssim \mu\bar{\theta}|\bar{y}-x_{t,s}^i|$, and hence this additional decay and the almost orthogonality arguments above can be integrated into the proof of \eqref{Jclaim} to obtain
$$
\|W^{j,l}_t (W^{j,l}_s)^*\|_{L^2 \to L^2} \lesssim \mu\bar{\theta}(1+\mu 2^{-2j}|t-s|)^{-2} \leq \mu 2^{-j}(1+\mu 2^{-2j}|t-s|)^{-2}.
$$

To show \eqref{twojdecay}, first consider  $t,s$ satisfying $\mu^{-1}2^{2j} < |t-s| \leq \mu^\beta 2^{-j+3}$ (note that this is nontrivial when $2^j < 2\mu^{\frac 13(1+\beta)}\approx 2^{N_\mu}$, a relevant consequence of the $j<N_\mu$ threshold discussed above).  We may assume $c_0$ in \eqref{c0} is sufficiently small  and use a linear approximation of the $n$-th component of $x_{t,s}\left(\bar{z},\eta^i\right)$ in \eqref{linearxnp} to obtain
\begin{equation}\label{linconseq}
\left|\left(x_{t,s}\right)_n\left(\bar{z},\eta^i\right)\right| \gtrsim 2^{-j}|t-s|
\end{equation}
since the $n$-th component of $\bar{z}$ vanishes.  Indeed, over this time scale, the error term is smaller than the linearization.

Now assume that $|t-s| \geq \mu^\beta 2^{-j+3}$.  Taking $\veps$, $c_0$ sufficiently small, we have that
$$
\left| (x_{t,s})_1(\bar{z},\eta^i) - z_1\right| = \left| \int_{s}^t \prtl_{\xi_1}p(r,\Theta_{r,s}(\bar{z},\eta^i))\,dr\right| \geq \frac 12 |t-s|
$$
Using that $y,z \in \supp(\chi_l)$ we have that $|y-z| \leq 2^{-j+1}\mu^\beta \leq \frac 14 |t-s|$ and hence we have the stronger bound
$$
\left| (x_{t,s})_1(\bar{z},\eta^i) - y_1\right| \geq \frac 12 |t-s| - |z_1-y_1| \geq \frac 14 |t-s|.
$$

\end{document}